\documentclass[reqno, 12pt]{amsart}

\usepackage{caption}
\usepackage{array}
\usepackage{amsmath}
\usepackage{amsfonts}
\usepackage{amssymb}
\usepackage{enumerate}
\usepackage{epsfig}
\usepackage{graphicx}
\usepackage{amsthm}
\usepackage{amsmath, amscd}
\usepackage{xy}
\xyoption{all}
\usepackage{color}
\usepackage{marvosym}
\usepackage{bm}
\usepackage{amsbsy}
\usepackage{hyperref}
\usepackage{tikz}
\usetikzlibrary{matrix,arrows,backgrounds}%, decorations, pathmorphing, positioning, fit}

\newtheorem{theorem}{Theorem}[subsection]
\newtheorem{theorem-section}{Theorem}[section]

\newtheorem{lemma-section}[theorem-section]{Lemma}
\newtheorem{proposition-section}[theorem-section]{Proposition}
\newtheorem{corollary}[theorem]{Corollary}
\newtheorem{corollary-section}[theorem-section]{Corollary}

\newtheorem{conjecture}[theorem]{Conjecture}

\theoremstyle{definition}
\newtheorem{definition}[theorem]{Definition}
\newtheorem{remark}[theorem]{Remark}
\newtheorem{remark-section}[theorem-section]{Remark}

\newcommand{\op}[1]{\operatorname{#1}}

\newcommand{\dcoh}[1]{\operatorname{D}(\operatorname{coh }#1)}
\newcommand{\dbcoh}[1]{\operatorname{D}^{\operatorname{b}}(\operatorname{coh }#1)}

\def\N{\op{\mathbb{N}}}
\def\Z{\op{\mathbb{Z}}}
\def\C{\op{\mathbb{C}}}

\def\O{\op{\mathcal{O}}}
\def\A{\op{\mathcal{A}}}
\def\P{\op{\mathbb{P}}}

\def\ra{\rightarrow}

\def\A{\mathop{\mathcal{A}}}

\def\AA{\op{\mathbb{A}}}
\def\GG{\op{\mathbb{G}}}

\def\cX{\mathcal{X}}
\def\cL{\mathcal{L}}

\def\cE{\mathcal{E}}
\def\cA{\mathcal{A}}
\def\cB{\mathcal{B}}
\def\cG{\mathcal{G}}
\def\cO{\mathcal{O}}
\def\cM{\mathcal{M}}

\def\cR{\mathcal{R}}

\def\cV{\mathcal{V}}
\def\cW{\mathcal{W}}

\def\cU{\mathcal{U}}
\def\cF{\mathcal{F}}
\def\cP{\mathcal{P}}

\def\cH{\mathcal{H}}

\def\gm{\mathbb{G}_m}

\def\iso{\cong}
\def\cbullet{\raisebox{0.02cm}{\scalebox{0.7}{$\bullet$}}}
\renewcommand{\t}[1]{\text{#1}}
\newcommand{\newterm}{\textsf}

\newcommand{\dpemodz}[1]{\op{D}_{\op{pe}}(\op{Mod}_{\Z}{#1})}
\newcommand{\dmodgraded}[2]{\op{D}(\op{mod}_{#1}{#2})}

\newcommand{\dpemodinftyZ}[1]{\op{D}_{\op{pe}}(\op{Mod}_{\infty,\Z}{#1})}

\title[On the Derived Categories of Degree $d$ Hypersurface Fibrations]{On the Derived Categories of Degree $d$ Hypersurface Fibrations}
\author[Ballard]{Matthew Ballard}
\address{
  \begin{tabular}{l}
   Matthew Ballard  \\ 
   \hspace{.1in} Universit\"at Wien, Fakult\"at f\"ur Mathematik,  Wien, \"Osterreich \\
   \hspace{.1in} Email: {\bf ballard@math.wisc.edu} \\
  \end{tabular}
}

\author[Deliu]{Dragos Deliu}
\address{
  \begin{tabular}{l}
   Dragos Deliu  \\ 
   \hspace{.1in} Universit\"at Wien, Fakult\"at f\"ur Mathematik,  Wien, \"Osterreich \\
   \hspace{.1in} Email: {\bf dragos.deliu@univie.ac.at} \\
  \end{tabular}
}

\author[Favero]{David Favero}
\address{
  \begin{tabular}{l}
   David Favero \\
   \hspace{.1in} Universit\"at Wien, Fakult\"at f\"ur Mathematik,  Wien, \"Osterreich \\
   \hspace{.1in} Email: {\bf favero@gmail.com} \\
  \end{tabular}
}

\author[Isik]{M. Umut Isik}
\address{
  \begin{tabular}{l}
   M. Umut Isik \\ 
   \hspace{.1in} Universit\"at Wien, Fakult\"at f\"ur Mathematik,  Wien, \"Osterreich \\
   \hspace{.1in} Email: {\bf mehmet.umut.isik@univie.ac.at} \\
  \end{tabular}
}

\author[Katzarkov]{Ludmil Katzarkov}
\address{
  \begin{tabular}{l}
   Ludmil Katzarkov \\
   \hspace{.1in} Universit\"at Wien, Fakult\"at f\"ur Mathematik,  Wien, \"Osterreich \\
   \hspace{.1in} Email: {\bf lkatzark@math.uci.edu} \\
  \end{tabular}
}

\numberwithin{equation}{section}

\begin{document}
\begin{abstract}
We provide descriptions of the derived categories of degree $d$ hypersurface fibrations which generalize a result of Kuznetsov for quadric fibrations and give a relative version of a well-known theorem of Orlov. Using a local generator and Morita theory, we re-interpret the resulting matrix factorization category as a derived-equivalent sheaf of dg-algebras on the base. Then, applying homological perturbation methods, we obtain a sheaf of $A_\infty$-algebras which gives a new description of homological projective duals for (relative) $d$-Veronese embeddings, recovering the sheaf of Clifford algebras obtained by Kuznetsov in the case when $d=2$. 
\end{abstract}
\maketitle

\section{Introduction} \label{sec: intro}

The purpose of this paper is to construct semi-orthogonal decompositions on derived categories of coherent sheaves of families of degree $d$ hypersurfaces and complete intersections, and to describe the semi-orthogonal pieces. The main inspirations for our results are two of the most appreciated theorems on derived categories: Orlov's theorem on the derived categories of hypersurfaces (\cite[Theorem 2.13]{Orl09}) and complete intersections (\cite[Proposition 2.16]{Orl09}), and Kuznetsov's theorem on quadric fibrations (\cite[Theorem 4.2]{Kuz05}).  

The first part of Orlov's theorem states that, given a hypersurface $X$ in $\P^{N-1}$ given as the zero-scheme of a polynomial $w$ of degree $d<N$, there is a semi-orthogonal decomposition 
       \begin{displaymath}
 	\dbcoh{X}  = \langle \O_Y(d-N+1), \ldots, \O_Y, \dcoh{[\AA^{N}/\mathbb{G}_m], w} \rangle,
       \end{displaymath}
where $\dcoh{[\AA^{N}/\mathbb{G}_m], w}$ denotes the $\C^\times$-equivariant (or $\Z$-graded) matrix factorization category of the Landau-Ginzburg pair $(\AA^N, w)$.

Kuznetsov's theorem on quadric fibrations states that given a quadric fibration $X$ defined as the zero-scheme of a section $s \in \Gamma(S, \op{Sym}^2 \mathcal E \otimes \mathcal L)$ in a projective bundle $q: \P(\mathcal E) \to S$, there is a semi-orthogonal decomposition
  \begin{gather*}
    \dbcoh{X} = \langle q^*\dbcoh{S}, \ldots, q^*\dbcoh{S} \otimes_{\mathcal O_X} \mathcal O_X(\op{rank} \mathcal E - 3), \dcoh{(S,\cB_0)} \rangle,
  \end{gather*}
where $\cB_0$ is a sheaf of even parts of Clifford algebras the Clifford relations of which are defined by second partial derivatives of $s$.

Kuznetsov's result can be interpreted as a relative version of Orlov's by replacing Orlov's matrix factorization category with a sheaf of non-commutative algebras on the base. This paper generalizes these theorems to families of hypersurfaces and complete intersections of general degree and connects these two theorems by showing that the sheaf of algebras $\cB_0$ is obtained from the matrix factorization category by homological perturbation methods applied to the endomorphism sheaf of dg-algebras of a local generator. 

Specifically, we prove in Corollary \ref{proposition: VGIT for degree d fibrations} that, given a family $X\subset \P(\cE)$ of hypersurfaces defined as the zero-scheme of a section $s \in \Gamma(S, \op{Sym}^d \mathcal E \otimes \mathcal L)$, where $d < \op{rank}\cE$, there is a semi-orthogonal decomposition 
  \begin{gather*}
    \dbcoh{X} = \langle q^*\dbcoh{S}, \ldots, q^*\dbcoh{S} \otimes_{\mathcal O_X} \mathcal O_X(\op{rank}\cE - d- 1), \dcoh{[\op{V}_S(\mathcal{E})/\mathbb{G}_m],w} \rangle,
  \end{gather*}
  where $w$ is contraction by $s$. Proposition \ref{proposition: Orlov rel CI} and Corollary \ref{proposition: VGIT for degree d fibrations} prove the analogous results for families of complete intersections as well as the cases when $d\geq N+1$. The technique for proving this theorem is to first pass from a space to the total space of a line bundle using the $\sigma$-model-Landau-Ginzburg-model correspondence of \cite{Isik, Shipman}, and then to vary Geometric Invariant Theory quotients (VGIT) i.e. birationally transform the total space of this line bundle \cite{BFK12, HL12, KawFF,VdB, Seg2, HW, DSe}.

We then study the category $\dcoh{[\op{V}_S(\mathcal{E})/\mathbb{G}_m],w}$ in more detail. Performing relative versions of the calculations in \cite{Se, Dyc11, Ef}, we prove that there is a local generator for the derived category of this gauged Landau-Ginzburg model and calculate its derived endomorphism sheaf of dg-algebras $\cB$ over the base $S$. We then use a version of the homological perturbation lemma to show that $\cB$ is quasi-isomorphic to a sheaf $\cA$ of minimal graded $A_\infty$-algebras. When $d=2$, $\cB$ and $\cA$ are derived-equivalent, which connects Orlov's and Kuznetsov's theorems as $\cA$ coincides with Kuznetsov's $\cB_0$. When $d>2$, the derived equivalence is implied by a technical conjecture (Conjecture \ref{conditionstar}).

Our main application is to Homological Projective Duality (HPD) \cite{KuzHPD} for the Veronese embeddings $\P(W) \rightarrow \P(S^dW)$ for $d\leq \op{dim}W$, which provides a powerful method of relating the derived categories of families of degree $d$ hypersurfaces and of intersections of degree $d$ hypersurfaces. In \cite{BDFIKhpd}, using new general methods for HPD for GIT quotients, it was proven that a Landau-Ginzburg pair $([\op{V}_S(\mathcal P)\times_S\P_S(S^d\mathcal P^*)/{\gm}],w)$ is a homological projective dual to the $d$-Veronese embedding of $\P(W)$. Combining this with the results described in the previous paragraph, we give a new description for the homological projective dual. More concretely, this Landau-Ginzburg pair is derived-equivalent to the pair $(\P(S^d W^*),\cA)$,
where $\cA$ is a $\Z$-graded sheaf of minimal $A_\infty$-algebras given by
 $$\mathcal A=\left(\bigoplus_{k\in \Z} u^k\cO_{\P(S^d W^*)}(k) \right)\otimes\Lambda^{\bullet} W^*,$$
 and higher products defined by explicit tree formulas, notably with 
 $$\mu^d(1\otimes v_{i_1},\ldots,1\otimes v_{i_d})=\frac{u}{d!}\frac{\partial^d w}{\partial x_{i_1}\ldots\partial x_{i_d}},$$
 where $\left\{ x_j \right\}$ denotes a basis of $W$ and $\left\{ v_j \right\}$ the corresponding basis of $\Lambda^{1} W^*$,
and $\mu^i = 0$ for $2<i<d$.
When $d>2$, this requires the assumption of the technical Conjecture \ref{conditionstar} mentioned above. When $d=2$, we recover the homological projective dual from \cite{Kuz05}.

\subsection*{Acknowledgments} We would like to thank Alexander Polishchuk for directing us toward \cite{Se} and for other helpful comments. We also thank David Ben-Zvi, Colin Diemer, Alexander Efimov, Kentaro Hori, Maxim Kontsevich, Alexander Kuznetsov, Yanki Lekili, Dmitri Orlov,  Pranav Pandit, Tony Pantev, and Anatoly Preygel for useful discussions. The authors were funded by NSF DMS 0854977 FRG, NSF DMS 0600800, NSF DMS 0652633 FRG, NSF DMS 0854977, NSF DMS 0901330, FWF P 24572 N25, by FWF P20778 and by an ERC Grant. The first author was funded, in addition, by NSF DMS 0838210 RTG.

\section{Background}\label{sec: background}
We begin by recalling some definitions and by reviewing part of the relationship between variations of GIT quotients \cite{Tha96, DH98} and derived categories, following \cite{BFK12} and \cite{BDFIKhpd}. We will use the same notation as in the background section of \cite{BDFIKhpd}.

Let $Q$ be a smooth and quasi-projective variety with the action of an affine algebraic group, $G$. Let $\cL$ be an invertible $G$-equivariant sheaf on $Q$ and let $w \in \op{H}^0(Q,\mathcal L)^G$ be a $G$-invariant section of $\cL$. 

\begin{definition}\label{deflgmodel}
 A \newterm{gauged Landau-Ginzburg model}, or \newterm{gauged LG model}, is the quadruple, $(Q, G, \cL, w)$, with $Q$, $G$, $\cL$, and $w$ as above. We shall commonly denote a gauged LG model by the pair $([Q/G],w)$. 
\end{definition}

Given a quasi-coherent $G$-equivariant sheaf, $\mathcal E$, we denote $\mathcal E \otimes \cL^n$ by $\mathcal E(n)$ and given a morphism, $f: \mathcal E \to \mathcal F$, we denote $f \otimes \op{Id}_{\cL^n}$ by $f(n)$. Following Eisenbud, \cite{EisMF}, we have the following definition.

\begin{definition}
 A \newterm{coherent factorization}, or simply a \newterm{factorization}, of a gauged LG model, $([Q/G],w)$, consists of a pair of coherent $G$-equivariant sheaves, $\mathcal E^{-1}$ and $\mathcal E^0$, and a pair of $G$-equivariant $\mathcal O_Q$-module homomorphisms,
\begin{align*}
 \phi^{-1}_{\mathcal E} &: \mathcal E^{0}(-1) \to \mathcal E^{-1} \\
 \phi^0_{\mathcal E} &: \mathcal E^{-1} \to \mathcal E^0
\end{align*}
such that the compositions, $\phi^0_{\mathcal E} \circ \phi^{-1}_{\mathcal E} : \mathcal E_{0}(-1) \to \mathcal E^0$ and  $\phi^{-1}_{\mathcal E}(1) \circ \phi_{\mathcal E}^0: \mathcal E^{-1} \to \mathcal E^{-1}(1)$, are multiplication by $w$.  We shall often simply denote the factorization $(\mathcal E^{-1}, \mathcal E^0, \phi_{\mathcal E}^{-1}, \phi_{\mathcal E}^0)$ by $\mathcal E$. The coherent $G$-equivariant sheaves, $\mathcal E^0$ and $\mathcal E^{-1}$, are called the \newterm{components of the factorization}, $\mathcal E$.

A \newterm{morphism of factorizations}, $g: \mathcal E \to \mathcal F$, is a pair of morphisms $(g^{-1}, g^{0})$ that commute with $\phi^{i}_{\cE}$ and $\phi^{i}_{\cF}$. Let $\op{coh}([Q/G],w)$ be the Abelian category of factorizations with coherent components.  

\end{definition}
 There is also a notion of a chain homotopy between morphisms in $\op{coh}([Q/G],w)$ and we let $K(\op{coh}[Q/G],w)$ be the corresponding homotopy category (c.f \cite{BDFIKhpd,BDFIK}), the category whose objects are factorizations and whose morphisms are homotopy classes of morphisms. 
In this category, one defines a natural translation autoequivalence and a natural cone construction. These induce the structure of a triangulated category on the homotopy category, $K(\op{Qcoh}[Q/G],w)$ \cite{Pos2, BDFIK}.

In order to derive $\op{coh}([Q/G],w)$, one takes a Verdier quotient by the correct substitute, in $\op{coh}([Q/G],w)$, for acyclic complexes. This was defined independently in \cite{Pos1}, \cite{OrlMF}.

\begin{definition}
 A factorization, $\mathcal A$, is called \newterm{totally acyclic} if it lies in the smallest thick subcategory of $K(\op{coh}[Q/G],w)$ containing all totalizations \cite[Definition 2.10]{BDFIK} of short exact sequences 
from $\op{coh}([Q/G],w)$. We let $\op{acycl}([Q/G],w)$ denote the thick subcategory of $K(\op{coh}[Q/G],w)$ consisting of totally acyclic factorizations.  
 
 The \newterm{absolute derived category of factorizations}, or the \newterm{derived category}, of the LG model $([Q/G],w)$, is the Verdier quotient,
\begin{displaymath}
 \op{D}(\op{coh}[Q/G],w) := K(\op{coh}[Q/G],w)/\op{acycl}([Q/G],w).
\end{displaymath}
 
 Abusing terminology, we say that $\mathcal E$ and $\mathcal F$ are \newterm{quasi-isomorphic} factorizations if they are isomorphic in the absolute derived category. 
\end{definition}

We now discuss the relationship between derived categories of factorizations and elementary wall crossings. 
Let $G$ be an algebraic group acting on a smooth algebraic variety $Q$, and let $\lambda: \mathbb{G}_m \to G$ be a one-parameter subgroup. We let $Z_{\lambda}^0$ be a choice of a connected component of the fixed locus of $\lambda$ on $Q$. Set
\begin{displaymath}
 Z_{\lambda} := \{ q \in Q \mid \lim_{t \to 0} \sigma(\lambda(t),q) \in Z_{\lambda}^0\}.
\end{displaymath}
The subvariety $Z_{\lambda}$ is called the \newterm{contracting locus} associated to $\lambda$ and $Z_{\lambda}^0$. If $G$ is Abelian, $Z_{\lambda}^0$ and $Z_{\lambda}$ are both $G$-invariant subvarieties. Otherwise, we must consider the orbits
\[
S_{\lambda} := G \cdot Z_\lambda ,
S_{\lambda}^0 := G \cdot Z_\lambda^0.
\]
Also, let
\begin{displaymath}
 Q_{\lambda} := Q \setminus S_{\lambda}.
\end{displaymath}

We will be interested in the case where $S_{\lambda}$ is a smooth closed subvariety satisfying a certain condition.  To state this condition we need the following group attached to any one-parameter subgroup
\begin{displaymath}
 P(\lambda) := \{ g \in G \mid \lim_{\alpha \to 0} \lambda(\alpha) g \lambda(\alpha)^{-1} \text{ exists}\}.
\end{displaymath}

\begin{definition} \label{definition: HKKN strat}
 Assume $Q$ is a smooth variety with a $G$-action. An \newterm{elementary HKKN stratification} of $Q$ is a disjoint union
\[
\mathfrak{K}: Q = Q_{\lambda} \sqcup S_{\lambda},
\]
obtained from the choice of a one-parameter subgroup $\lambda: \mathbb{G}_m \to G$, together with the choice of a connected component, denoted $Z_{\lambda}^0$, of the fixed locus of $\lambda$ such that
 \begin{itemize}
   \item $S_{\lambda}$ is closed in $X$.
 \item The morphism,
\begin{align*}
 \tau_{\lambda}: [(G \times Z_{\lambda})/P(\lambda)] & \to S_{\lambda} \\
 (g,z) & \mapsto g \cdot z
\end{align*}
is an isomorphism where $p \in P(\lambda)$ acts by
\begin{displaymath}
 (p,(g,z)) \mapsto (gp^{-1},p \cdot z).
\end{displaymath}
 \end{itemize}
\end{definition}

We will need to attach an integer to an elementary HKKN stratification. We restrict the relative canonical bundle $\omega_{S_{\lambda}|Q}$ to any fixed point $q \in Z_{\lambda}^0$.  This yields a one-dimensional vector space which is equivariant with respect to the action of $\lambda$.  

\begin{definition}
 The \newterm{weight of the stratum} $S_{\lambda}$ is the $\lambda$-weight of $\omega_{S_{\lambda}/ Q}|_{Z_{\lambda}^0}$. It is denoted by $t(\mathfrak{K})$.
\end{definition}

Furthermore, given a one parameter subgroup $\lambda$ we may also consider its composition with inversion
\[
-\lambda(t) := \lambda(t^{-1}) = \lambda(t)^{-1},
\]
and ask whether this provides an HKKN stratification as well. This leads to the following definition.

\begin{definition}
 An \newterm{elementary wall-crossing}, $(\mathfrak{K}^+,\mathfrak{K}^-)$, is a pair of elementary HKKN stratifications,
 \begin{align*}
 Q = & Q_{\lambda} \sqcup S_{\lambda},\\
 Q = & Q_{-\lambda} \sqcup S_{-\lambda},
 \end{align*}
 such that $Z^0_\lambda = Z^0_{-\lambda}$. We often let $Q_+ := Q_{\lambda}$ and $Q_{-} := Q_{-\lambda}$.
\end{definition}

Let $C(\lambda)$ denote the centralizer of the $1$-parameter subgroup $\lambda$. For an elementary wall-crossing set
\[
\mu = - t(\mathfrak{K}^+) + t(\mathfrak{K}^-).
\]

\begin{theorem-section} \label{thm: bfkvgitlg}
 Let $Q$ be a smooth, quasi-projective variety equipped with the action of a reductive linear algebraic group, $G$. Let $w \in \op{H}^0(Q,\mathcal L)^G$ be a $G$-invariant section of a $G$-invertible sheaf, $\mathcal L$. Suppose we have an elementary wall-crossing, $(\mathfrak{K}^+,\mathfrak{K}^-)$, 
 \begin{align*}
  Q & = Q_+ \sqcup S_{\lambda} \\
  Q & = Q_- \sqcup S_{-\lambda},
 \end{align*}
and assume that $\mathcal L$ has weight zero on $Z_{\lambda}^0$ and that $S^0_{\lambda}$ admits a $G$ invariant affine open cover. Fix any $D\in \Z$. 
 
 \begin{enumerate}
  \item If $\mu > 0$, then there are fully-faithful functors,
  \begin{displaymath}
   \Phi^+_D: \dcoh{[Q_-/G],w|_{Q_-}} \to \dcoh{[Q_+/G],w|_{Q_+}},
  \end{displaymath}
  and, for $-t(\mathfrak{K}^-) + D \leq j \leq -t(\mathfrak{K}^+) + D -1$,
  \begin{displaymath}
   \Upsilon_j^+: \dcoh{[Z_{\lambda}^0/C(\lambda)],w_{\lambda}}_j \to \dcoh{[Q_+/G],w|_{Q_+}},
  \end{displaymath}
  and a semi-orthogonal decomposition,
  \begin{displaymath}
   \dcoh{[Q_+/G],w|_{Q_+}} = \langle  \Upsilon^+_{-t(\mathfrak{K}^-)+D}, \ldots, \Upsilon^+_{-t(\mathfrak{K}^+)+D-1}, \Phi^+_D  \rangle.
  \end{displaymath}
  \item If $\mu = 0$, then there is an exact equivalence,
  \begin{displaymath}
   \Phi^+_D: \dcoh{[Q_-/G],w|_{Q_-}} \to \dcoh{[Q_+/G],w|_{Q_+}}.
  \end{displaymath}
  \item If $\mu < 0$, then there are fully-faithful functors,
  \begin{displaymath}
   \Phi^-_D: \dcoh{[Q_+/G],w|_{Q_+}} \to \dcoh{[Q_-/G],w|_{Q_-}},
  \end{displaymath}
  and, for $-t(\mathfrak{K}^+) + D \leq j \leq -t(\mathfrak{K}^-) + D -1$,
  \begin{displaymath}
   \Upsilon_j^-: \dcoh{[Z_{\lambda}^0/C(\lambda)],w_{\lambda}}_j \to \dcoh{[Q_-/G],w|_{Q_-}},
  \end{displaymath}
  and a semi-orthogonal decomposition,
  \begin{displaymath}
   \dcoh{[Q_-/G],w|_{Q_-}} = \langle  \Upsilon^-_{-t(\mathfrak{K}^+)+D}, \ldots, \Upsilon^-_{-t(\mathfrak{K}^-)+D-1}, \Phi^-_D \rangle.
  \end{displaymath}
 \end{enumerate}
\end{theorem-section}

\begin{proof}
 This is \cite[Theorem 3.5.2]{BFK12}.
\end{proof}

The categories, $\dcoh{[Z_{\lambda}^0/C(\lambda)],w_{\lambda}}_j$, appearing in Theorem \ref{thm: bfkvgitlg} are the full subcategories consisting of objects of $\lambda$-weight $j$ in $\dcoh{[Z_{\lambda}^0/C(\lambda)],w_{\lambda}}$. For more details, we refer the reader to \cite{BFK12}. In our situation, we will only need the conclusion of the following lemma. We set 
\begin{displaymath}
 Y_{\lambda} : = [ Z_{\lambda}^0 / (C(\lambda)/\lambda) ].
\end{displaymath}

\begin{lemma-section} \label{lemma: wall compositions are same}
 We have an equivalence,
 \begin{displaymath}
  \dcoh{Y_{\lambda},w_{\lambda}} \cong \dcoh{[Z_{\lambda}^0/C(\lambda)],w_{\lambda}}_0.
 \end{displaymath}

 Further, assume that there there is a character, $\chi: C(\lambda) \to \mathbb{G}_m$, such that
 \begin{displaymath}
  \chi \circ \lambda(t) = t^l.
 \end{displaymath}
 Then, twisting by $\chi$ provides an equivalence,
 \begin{displaymath}
  \dcoh{[Z^0_{\lambda}/C(\lambda)],w_{\lambda}}_r  \cong \dcoh{[Z^0_{\lambda}/C(\lambda)],w_{\lambda}}_{r+l},
 \end{displaymath}
 for any $r \in \Z$.
\end{lemma-section}

\begin{proof}
 This is Lemma 3.4.4 of \cite{BFK12}; we give the very simple and short proof here. A quasi-coherent sheaf on $Y_{\lambda}$ is a quasi-coherent $C(\lambda)$-equivariant sheaf on $Z_{\lambda}^0$ for which $\lambda$ acts trivially, i.e. of $\lambda$-weight zero.  For the latter statement just observe that twisting with $\chi$ is an autoequivalence of $\dbcoh{[Z_{\lambda}^0/C(\lambda)]}$ which brings range to target and its inverse does the reverse.
\end{proof}

\section{Relative version of Orlov's theorem} \label{sec: deg d fib via VGIT}

\begin{theorem-section}\label{thm: orlov}
  Let $X$ be a hypersurface of degree $d$ given as the zero-scheme of $w \in \Gamma(\P(V),\O_{\P(V)}(d))$. Let $N=\op{dim}V$.
 \begin{enumerate}
  \item If $d<N$, then there is a semi-orthogonal decomposition
       \begin{displaymath}
 	\dbcoh{X}  = \langle \O_Y(d-N+1), \ldots, \O_Y, \dcoh{[\AA^{N}/\mathbb{G}_m], w} \rangle.
       \end{displaymath}
  \item If $d=N$, there is an equivalence 
       \begin{displaymath}
 	\dbcoh{X} \cong \dcoh{[\AA^N/\mathbb{G}_m] , w}  
       \end{displaymath}
  \item If $d>N$, there is a semi-orthogonal decomposition
       \begin{displaymath}
 	\dcoh{[\AA^N/\mathbb{G}_m], w}  = \langle k(N-d+1),\ldots, k, \dbcoh{X} \rangle.
       \end{displaymath}    
 \end{enumerate}
\end{theorem-section}
 
\begin{proof}
 This is \cite[Theorem 2.13]{Orl09}.
\end{proof}

We will generalize this statement and the complete intersection version of it (\cite[Proposition 2.16]{Orl09}) to families of complete intersections over a base. To this end, let $S$ be a smooth, connected variety, $\mathcal E$ be a locally-free coherent sheaf of rank $n$ on $S$ and let $\mathcal L_i$ be invertible sheaves on $S$, for $1 \leq i \leq c$. Set $\mathcal U = \bigoplus_i \mathcal L_i$.  Let
\begin{displaymath}
 Q := \op{V}_S\left(\mathcal E \oplus \mathcal U \right).
\end{displaymath}
Let $q: \P(\mathcal E) \to S$ be the projection. Choose sections $s_i \in \Gamma(S, S^{d_i} \mathcal E \otimes \mathcal L_i)$, let $\tilde{s}_i$ be the corresponding sections in  $\Gamma(\P(\mathcal E),\mathcal O_{\P(\mathcal E)}(d_i) \otimes q^*\mathcal L_i)$ and let $X$ be the zero locus of $(\tilde{s}_1,\ldots,\tilde{s}_c)$ in $\P(\mathcal E)$. Let $w_i$ be the associated regular functions on $Q$. Let $w = \sum_i w_i$. Consider the $\mathbb{G}_m^2$-action on $Q$ given by 
\begin{align*}
 \sigma: \mathbb{G}_m^2 \times Q & \to Q \\
 (\alpha_1,\alpha_2, (e,\oplus_i p_i, s)) & \mapsto ( \alpha_1^{-1} e,\oplus_i \alpha_1^{d_i}\alpha_2^{-1} p_i,s). 
\end{align*}
The function $w$ becomes invariant with respect to the $\mathbb{G}_m$-action given by the one-parameter subgroup $\lambda(\alpha) = (\alpha,1)$. It is semi-invariant of weight $1$ for the other $\mathbb{G}_m$-action. 
 
The one-parameter subgroup $\lambda$ induces an elementary wall crossing (as in Definition \ref{definition: HKKN strat}). To see this, first observe that the fixed locus, $Z_{\lambda}^0$, is the zero section, $\mathbf{0}_Q$, of $Q$. We have 
\begin{align*}
 S_{\lambda} = Z_{\lambda} & = \mathbf{0}_{\op{V}_S(\cE)}\times_S \op{V}_S(\mathcal U) \\
 S_{-\lambda} = Z_{-\lambda} & = \op{V}_S(\mathcal E)\times_S \mathbf{0}_{\op{V}_S(\cU)}.
\end{align*}
Both are closed. The fact that  $\tau_{\lambda}$ is an isomorphism is trivial when the ambient group is Abelian.
 
 We then have 
 \begin{align*}
   Q_+ & =  \left( \op{V}_S(\mathcal E) \setminus \mathbf{0}_{\op{V}_S(\cE)} \right) \times_S \op{V}_S(\mathcal U) \\
  Q_- & =  \op{V}_S(\mathcal E)\times_S \left( \op{V}_S(\mathcal U) \setminus \mathbf{0}_{\op{V}_S(\mathcal U)} \right) .
 \end{align*}
 And, $\mu = \op{rank} \mathcal E - \sum d_i$. Note that $C(\lambda) = \mathbb{G}_m^2$ so $C(\lambda)/\lambda \cong \mathbb{G}_m$. We have an equivalence
 \begin{displaymath}
  \dcoh{[Z_{\lambda}^0/(C(\lambda)/\lambda)],w_{\lambda}} = \dcoh{[S/\mathbb{G}_m],0} \cong \dbcoh{S}.
 \end{displaymath}
 In this example, we observe using \cite[Section 3.4]{BFK12} that $\Upsilon_l^{\pm}$ is the functor 
 \begin{displaymath}
  i_{\pm}^* \circ (j_\mp)_* \circ  \pi_\mp^* : \dcoh{[Z_{\lambda}^0/C(\lambda)],w_{\lambda}}_l \to \dcoh{[Q_{\pm}/\mathbb{G}_m^2],w}.
 \end{displaymath}
 where
 \begin{align*}
    S \overset{\pi_+}{\leftarrow} \op{V}_S(\mathcal E) \overset{j_+}{\to} Q \overset{i_-}{\leftarrow} Q_-\\
    S \overset{\pi_-}{\leftarrow} \op{V}_S(\mathcal U) \overset{j_-}{\to} Q \overset{i_+}{\leftarrow} Q_+ 
 \end{align*}
 are projections and inclusions. Let $\pi: Q \to \P(\mathcal E)$ be the projection. The pullback, $\pi^* \mathcal O_{\P(\mathcal E)}(1)$ has weight $1$ with respect to $\lambda$. Therefore, we have equivalences between the essential images 
 \begin{align*}
  \op{EssIm} \Upsilon_0^+ \otimes \mathcal \pi^*\mathcal O_{\P(\mathcal E)}(l) & \cong \op{EssIm} \Upsilon_l^+ \\
  \op{EssIm} \Upsilon_0^- \otimes \mathcal \pi^*\mathcal O_{\P(\mathcal E)}(-l) & \cong \op{EssIm} \Upsilon_l^-.
 \end{align*}
 Applying Theorem \ref{thm: bfkvgitlg}, we have the following statements.
 
 \begin{itemize}
  \item If $\mu = \op{rank} \mathcal E - \sum_i d_i > 0$, there is a semi-orthogonal decomposition 
  \begin{gather*}
   \dcoh{[Q_+/\mathbb{G}_m^2],w} \cong \langle \Upsilon_0^+\dcoh{[S/\mathbb{G}_m],0}, \Upsilon_0^+\dcoh{[S/\mathbb{G}_m],0} \otimes \pi^*\mathcal O_{\P(\mathcal E)}(1), \ldots, \\ \Upsilon_0^+\dcoh{[S/\mathbb{G}_m],0} \otimes \pi^*\mathcal O_{\P(\mathcal E)}(\mu-1), \Phi^+ \dcoh{[Q_-/\mathbb{G}_m^2],w} \rangle.
  \end{gather*}
  \item If $\mu = \op{rank} \mathcal E - \sum_i d_i = 0$, there is an equivalence 
  \begin{gather*}
   \dcoh{[Q_+/\mathbb{G}_m^2],w} \cong \dcoh{[Q_-/\mathbb{G}_m^2],w}.
  \end{gather*}
  \item If $\mu = \op{rank} \mathcal E - \sum_i d_i < 0$, there is a semi-orthogonal decomposition 
  \begin{gather*}
   \langle \Upsilon_0^-\dcoh{[S/\mathbb{G}_m],0}, \Upsilon_0^-\dcoh{[S/\mathbb{G}_m],0} \otimes \pi^*\mathcal O_{\P(\mathcal E)}(-1), \ldots, \\ \Upsilon_0^-\dcoh{[S/\mathbb{G}_m],0} \otimes \pi^*\mathcal O_{\P(\mathcal E)}(\mu+1), \Phi^- \dcoh{[Q_+/\mathbb{G}_m^2],w} \rangle \cong \dcoh{[Q_-/\mathbb{G}_m^2],w} .
  \end{gather*}
 \end{itemize}
 
 There is an isomorphism
 \begin{displaymath}
  [Q_+/\lambda] \cong \op{V}_{\P(\mathcal E)} \left( \bigoplus_i \mathcal O_{\P(\mathcal E)}(d_i) \otimes_{\mathcal O_{\P(\mathcal E)}} q^*\mathcal L_i \right) 
 \end{displaymath}
 under which $w$ corresponds to the regular function determined by $\oplus_i \tilde{s}_i$. We now need the following theorem.
 
\begin{theorem-section}  \label{thm: isik}
Let $Y$ be the zero-scheme of a section $s\in \Gamma(X,\cE)$ of a locally-free sheaf of finite rank $\cE$ on a smooth variety $X$. Assume that that $s$ is a regular section, i.e. $\op{dim}Y = \op{dim}X - \op{rank} \cE$. Then, there is an equivalence of triangulated categories
\begin{equation*}
\dbcoh{Y} \iso \dcoh{[\op{V}(\cE)/\mathbb{G}_m], w}
\end{equation*}
where $w$ is the regular function determined by $s$ under the natural isomorphism
\begin{displaymath}
\Gamma(\op{V}(\cE),\mathcal O) \cong \Gamma(X, \op{Sym} \mathcal E)
\end{displaymath}
and $\mathbb{G}_m$ acts by dilation on the fibers.
\end{theorem-section}

\begin{proof}
This is \cite[Theorem 3.6]{Isik} or \cite[Theorem 3.4]{Shipman}.
\end{proof}
Assuming further that $\oplus_i \tilde{s}_i$ is a regular section (which implies that $X$ is of pure codimension $c$) and applying Theorem \ref{thm: isik}, we have an equivalence
 \begin{displaymath}
  \Psi: \dcoh{[Q_+/\mathbb{G}_m^2],w} \cong \dcoh{[\op{V}_{\P(\mathcal E)} \left( \bigoplus_i \mathcal O_{\P(\mathcal E_i)}(d_i) \otimes_{\mathcal O_{\P(\mathcal E)}} q^*\mathcal L_i \right)/\mathbb{G}_m],w} \cong \dbcoh{X}.
 \end{displaymath}
 Under this equivalence, we have an isomorphism of functors
 \begin{displaymath}
  \Psi \circ (\Upsilon_0^+ \otimes \pi^*\mathcal O_{\P(\mathcal E)}(l)) \cong q^* \otimes \mathcal O_X(l) : \dbcoh{S} \to \dbcoh{X}
 \end{displaymath}
 where $q: X \subset \P(\mathcal E) \to S$ also denotes the projection from $X$. Thus, we get the following
 
 \begin{proposition-section} \label{proposition: Orlov rel CI}
  With the assumptions above, we have the following
 \begin{itemize}
  \item If $\mu = \op{rank} \mathcal E - \sum_i d_i > 0$, there is a semi-orthogonal decomposition 
  \begin{gather*}
   \dbcoh{X} \cong \langle q^*\dbcoh{S}, \ldots, q^*\dbcoh{S} \otimes_{\mathcal O_X} \mathcal O_X(\mu - 1), \Phi^+ \dcoh{[Q_-/\mathbb{G}_m^2],w} \rangle.
  \end{gather*}
  \item If $\mu = \op{rank} \mathcal E - \sum_i d_i = 0$, there is an equivalence 
  \begin{gather*}
   \dbcoh{X} \cong \dcoh{[Q_-/\mathbb{G}_m^2],w}.
  \end{gather*}
  \item If $\mu = \op{rank} \mathcal E - \sum_i d_i < 0$, there is a semi-orthogonal decomposition 
  \begin{gather*}
   \langle \Upsilon_0^-\dcoh{[S/\mathbb{G}_m],0}, \Upsilon_0^-\dcoh{[S/\mathbb{G}_m],0} \otimes \pi^*\mathcal O_{\P(\mathcal E)}(-1), \ldots, \\ \Upsilon_0^-\dcoh{[S/\mathbb{G}_m],0} \otimes \pi^*\mathcal O_{\P(\mathcal E)}(\mu+1), \Phi^- \dbcoh{X} \rangle \cong \dcoh{[Q_-/\mathbb{G}_m^2],w} .
  \end{gather*}
 \end{itemize}
  
 \end{proposition-section}

 \begin{proof}
   This follows from Theorem \ref{thm: bfkvgitlg} and Theorem \ref{thm: isik}, as discussed above.
 \end{proof}

In the case when $X$ is a family of degree $d$ hypersurfaces, $\mathcal{U}$ is a invertible sheaf and we can write
\begin{equation*}
  [Q_- / \mathbb{G}_{m}^{2}] = [ \op{V}_S(\mathcal E)\times_S \left( \op{V}_S(\mathcal U) \setminus 0 \right) /\mathbb{G}_m^{2}] \iso [\op{V}_S(\mathcal E)/\mathbb{G}_m].
\end{equation*}
where the action of $\mathbb{G}_m$ is by fiber-wise dilation. 
 In the coming sections, we will restrict our attention only to this case. We record this in the following corollary.   
 
 \begin{corollary-section} \label{proposition: VGIT for degree d fibrations}
  Let $s \in \Gamma(S, \op{Sym}^d \mathcal E \otimes \mathcal L)$ and let $X \subset \P(\mathcal E)$ be the associated degree $d$ hypersurface fibration over $S$, with structure map, $q: \P(\mathcal E) \to S$. 
  \begin{itemize}
  \item If $\mu = \op{rank} \mathcal E - d > 0$, there is a semi-orthogonal decomposition 
  \begin{gather*}
    \dbcoh{X} = \langle q^*\dbcoh{S}, \ldots, q^*\dbcoh{S} \otimes_{\mathcal O_X} \mathcal O_X(\mu - 1), \dcoh{[\op{V}_S(\mathcal{E})/\mathbb{G}_m],w} \rangle.
  \end{gather*}
  \item If $\mu = \op{rank} \mathcal E - d = 0$, there is an equivalence 
  \begin{gather*}
   \dbcoh{X} \cong \dcoh{[\op{V}_S(\mathcal{E})/\mathbb{G}_m],w}.
  \end{gather*}
  \item If $\mu = \op{rank} \mathcal E - d < 0$, there is a semi-orthogonal decomposition 
  \begin{gather*}
   \langle \Upsilon_0^-\dcoh{[S/\mathbb{G}_m],0}, \Upsilon_0^-\dcoh{[S/\mathbb{G}_m],0} \otimes \pi^*\mathcal O_{\P(\mathcal E)}(-1), \ldots, \\ \Upsilon_0^-\dcoh{[S/\mathbb{G}_m],0} \otimes \pi^*\mathcal O_{\P(\mathcal E)}(\mu+1), \dbcoh{X} \rangle \cong \dcoh{[\op{V}_S(\mathcal{E})/\mathbb{G}_m],w} .
  \end{gather*}
 \end{itemize}
 \end{corollary-section}
 
 \begin{proof}
  This is a special case of Proposition~\ref{proposition: Orlov rel CI}.
 \end{proof}

 \section{A local generator and Morita theory}\label{sec:morita}

 In this section, we continue within the setting presented in Section \ref{sec: deg d fib via VGIT} with $c=1$ and show that the gauged LG model $([\op{V}_S(\mathcal{E})/\mathbb{G}_m],w)$ is derived-equivalent to the pair $(S,\mathcal{B}_w)$ where $\mathcal{B}_w$ is an equivariant sheaf of dg-algebras.  However, it will be more convenient to work with the isomorphic gauged LG model $([Q_-/\mathbb{G}_m^2],w)=\left( [\left( \op{V}(\mathcal U) \setminus \mathbf{0}_{\op{V}_{S}(\cU)} \right) \times_S \op{V}(\mathcal E) / \mathbb{G}_m^2],w\right)$. The sheaf $\mathcal{B}_w$ will be the derived equivariant endomorphism sheaf of algebras of a ``local generator'' $\mathcal{G}$ of the derived category of this latter LG model. 

Throughout this section, we will make the further assumption that the subvariety defined by $w=0$ in $\P_S(\cE)$ is smooth.

Let us recall our setup. We work over a base $S$ which is a smooth connected variety and $\mathcal E$ is a locally-free sheaf $S$.  Meanwhile, we have specialized to the case where  $\mathcal U$ is an invertible sheaf on $S$. We have a $\mathbb{G}_m^2$-action on
\begin{displaymath}
  Q = \op{V}_S\left(\mathcal E \right) \times_S \op{V}_S\left(\mathcal U \right),
\end{displaymath}
given by 
\begin{align*}
 \sigma: \mathbb{G}_m^2 \times Q & \to Q \\
 (\alpha_1,\alpha_2, (e,p, s)) & \mapsto (\alpha_1^{-1} e, \alpha_1^{d}\alpha_2^{-1} p ,s). 
\end{align*}
We also have a projection $q: \P(\mathcal E) \to S$ and have denoted by $X$, which is assumed to be smooth, the zero locus in $\P(\mathcal E)$of a section of $\mathcal O_{\P(\mathcal E)}(d_i) \otimes_{\mathcal O_{\P(\mathcal E)}} q^*\mathcal U$ corresponding to the regular function $w$ on $Q$.  

Recall that
\begin{displaymath}
  Q_- =  \op{V}_S(\mathcal E)\times_S (\op{V}_S(\mathcal U)\setminus \mathbf{0}_{\op{V}_S(\mathcal U)}). 
\end{displaymath}
We will replace the category $\dcoh{[Q_-/\mathbb{G}_m^2],w}$ by the derived category of a $\mathbb{G}_m$-equivariant sheaf of dg-algebras over $S$. Let $\pi: Q \to S$ denote the projection. We shall also denote the projection, $\pi: Q_- \to S$. Recall that 
\begin{displaymath}
 \pi_* \mathcal O_{Q} \cong \op{Sym}(\mathcal E)\otimes_{\cO_S} \op{Sym}(\mathcal U)
\end{displaymath}
and define
\begin{displaymath}
 \cR:=\pi_* \mathcal O_{Q_-} \cong \op{Sym}(\mathcal E)\otimes_{\cO_S} \op{Sym}(\mathcal U,\cU^{-1}),
\end{displaymath} 
where we used the notation 
$$\op{Sym}(\mathcal U,\cU^{-1}) := \bigoplus_{k\in\Z}\cU^{\otimes k}.$$
Since $\pi$ is affine, we have an equivalence 

\begin{equation}\label{eqn: affine}
  \dcoh{[Q_-/\mathbb{G}_m^2],w}\cong\dmodgraded{\Z^2}{\cR,w}
\end{equation}
of $\dcoh{[Q_-/\mathbb{G}_m^2],w}$ with $\dmodgraded{\Z^2}{\cR,w}$, the category of $\Z^2$-graded coherent factorizations over $\cR$. From now on we will be working in this latter category.

Let $Y = \mathbf{0}_{\op{V}_S(\cE)} \times_S (\op{V}_S(\mathcal U) \setminus \mathbf{0}_{\op{V}_S(\cU)}) \subset Q_-$. 
We have the object
\begin{center}
\begin{tikzpicture}[description/.style={fill=white,inner sep=2pt}]
\matrix (m) [matrix of math nodes, row sep=3em, column sep=3em, text height=1.5ex, text depth=0.25ex]
{  0 & \mathcal O_{Y} \\ };
\path[->,font=\scriptsize]
(m-1-1) edge[out=30,in=150] node[above] {$0$} (m-1-2)
(m-1-2) edge[out=210, in=330] node[below] {$0$} (m-1-1);
\end{tikzpicture}
\end{center}
of $\dcoh{[Q_-/\mathbb{G}_m^2],w}$ which corresponds to the object
$$\mathcal G:=(0,\op{Sym}(\cU,\cU^{-1}),0,0)$$ in $\dmodgraded{\Z^2}{\cR,w}$.

\begin{proposition-section}\label{prop:localgenerator}
  The objects $\left\{ \mathcal G \otimes_{\mathcal O_S} \mathcal L \right\}$, with $\mathcal L$ invertible $\mathbb{G}_m$-equivariant sheaves on $S$, generate $\dmodgraded{\Z^2}{\cR,w}$ i.e. for any non-zero equivariant complex $\cM$ of $\cO_S$-modules, there exists $\cL$ such that $\op{Ext^n}(\cL,\cM)\neq 0$, for some integer $n$.
\end{proposition-section}

\begin{proof}
 Let $C$ be the zero locus of $w$ inside the relative spectrum $\underline{\op{Spec}} \ \mathcal R$.  Then, there is an essentially surjective functor, 
 \[
 \Upsilon: \op{D}_{\op{sg}}^{\mathbb{G}_m^2}(C) \to \dcoh{[Q_-/\mathbb{G}_m^2],w} \cong \dmodgraded{\Z^2}{\cR,w},
 \]
 by \cite[Proposition 3.64]{BFK11}. The category, $\op{D}_{\op{sg}}^{\mathbb{G}_m^2}(C)$, is generated by objects pushed forward from the singular locus of $C$ by \cite[Corollary 4.14]{BFK11}. Since we have assumed that $X$ is smooth, the singular locus of $C$ is exactly $Y = \underline{\op{Spec}} \ \op{Sym}(\mathcal U,\cU^{-1})$. The $\mathbb{G}_m^2$-equivariant derived category of $Y$ is equivalent to the $\mathbb G_m$-equivariant derived category of $S$ under pullback via the projection 
 \begin{displaymath}
  \rho: Y \to S.
 \end{displaymath}
The composition of pullback to $\dbcoh{Y}$, push-forward to $\op{D}_{\op{sg}}^{\mathbb{G}_m^2}(C)$ and $\Upsilon$, is essentially surjective and maps an invertible equivariant sheaf $\cL$ on $S$ to $\cG\otimes_{\cO_S} \cL$.  Therefore it suffices to show that the the derived category of $[S/\mathbb{G}_m]$ is generated by invertible sheaves. 

 Let $\cM$ be an equivariant complex of $\cO_S$-modules such that $\cH^{n}\cM \neq 0$ for some $n$. By \cite[Lemma 2.10]{Tho2}, there is an equivariant invertible sheaf $\cL^{-1}$ and a non-zero invariant section $\mu\in \Gamma(S, \op{ker}d^n_{\cM}\otimes_{\cO_S}\cL^{-1})^{\mathbb{G}_m}$ whose image in $\Gamma(S, \cH^n\cM\otimes_{\cO_{S}\cL^-1})^{\mathbb{G}_m}$ is also non-zero. Here, in addition to the Lemma in \emph{loc. cit.}, we are using the resolution property by direct sums of line bundles for $\GG_m$-equivariant sheaves over a point. This defines a map $\cO_S \rightarrow \cM\otimes \cL^{-1}[n]$ and therefore a map $\cL \rightarrow \cM[n] $ which is non-trivial in cohomology. 
% By \cite[Lemma 2.10]{Tho2}, Thus, $\dcoh{[Q_-/\mathbb{G}_m^2],w}$ is generated by factorizations of the form 
% \begin{center}
% \begin{tikzpicture}[description/.style={fill=white,inner sep=2pt}]
% \matrix (m) [matrix of math nodes, row sep=3em, column sep=3em, text height=1.5ex, text depth=0.25ex]
% {  0 & \rho^* \mathcal L \\ };
% \path[->,font=\scriptsize]
% (m-1-1) edge[out=30,in=150] node[above] {$0$} (m-1-2)
% (m-1-2) edge[out=210, in=330] node[below] {$0$} (m-1-1);
% \end{tikzpicture}
% \end{center}
% with $\mathcal L$ invertible. 
\end{proof}

Define $$\mathcal B_w:=\bigoplus_{i\in\Z}{\mathbf{R}\mathcal Hom}_{\mathcal R,w,\Z^2}(\mathcal G(i,0),\mathcal G),$$
and the functor,
\begin{align*}
  F: \dmodgraded{\Z^2}{\cR,w} & \to \op{D}(\op{Mod}_{\Z} \mathcal B_w) \\
 \mathcal E & \mapsto \bigoplus_{i\in\Z}{\mathbf{R}\mathcal Hom}_{\mathcal R,w,\Z^2}(\mathcal G(i,0),\mathcal E).
\end{align*}

\begin{proposition-section}\label{prop: ainftyequiv}
 The functor, $F$, is fully-faithful.
\end{proposition-section}

\begin{proof}
% Let us first recall, from e.g. \cite[Section 1]{Riche}, that we have a sheaf-Hom for sheaves of $\mathcal B_w$-modules, $\mathcal Hom_{\mathcal B_w}$, which outputs a chain complex of $\mathcal O_S$-modules. 
 We first check that $F$ induces natural quasi-isomorphisms of chain complexes 
 \begin{displaymath}
  \mathbf{R} \mathcal Hom_{\mathcal R,w,\Z^2} (\mathcal G,\mathcal E) \to \mathbf{R} \mathcal Hom_{\mathcal B_w,\Z}(F (\mathcal G), F(\mathcal E))
 \end{displaymath}
 for any object $\mathcal E \in \dmodgraded{\Z^2}{\cR,w}$. Since $\mathcal B_w := F (\mathcal G)$, we have 
 \begin{gather*}
  \mathbf{R} \mathcal Hom_{\mathcal B_w,\Z}(F (\mathcal G), F(\mathcal E)) = \mathcal Hom_{\mathcal B_w,\Z}(\mathcal B_w, F(\mathcal E)) \\ \cong (F(\mathcal E))_0 := \left(\bigoplus_{i\in\Z}{\mathbf{R}\mathcal Hom}_{\mathcal R,w,\Z^2}(\mathcal G(i,0),\mathcal E) \right)_0 = \mathbf{R}\mathcal Hom_{\mathcal R,w,\Z^2}(\mathcal G,\mathcal E).
 \end{gather*}

  Therefore, 
 \begin{align*}
  \mathbf{R} \mathcal Hom_{\mathcal R,w,\Z^2} (\mathcal G \otimes_{\mathcal O_S} \mathcal L,\mathcal E) & \cong \mathbf{R}\mathcal Hom_{\mathcal R,w,\Z^2} (\mathcal G, \mathcal E \otimes_{\mathcal O_S} \mathcal L^{-1}) \\
  & \cong  \mathbf{R} \mathcal Hom_{\mathcal B_w,\Z} (F( \mathcal G ) , F( \mathcal E \otimes_{\mathcal O_S} \mathcal L^{-1}))\\
  & \cong \mathbf{R}\mathcal Hom_{\mathcal B_w,\Z} (F (\mathcal G \otimes_{\mathcal O_S} \mathcal L), F(\mathcal E)),
 \end{align*}
 where the last step follows from the fact that $F$ commutes with $\cbullet \otimes_{\cO_S}\cL$. 

 Applying $ \mathbf{R} \Gamma$ shows that $F$ is fully-faithful on the smallest thick subcategory of $\dmodgraded{\Z^2}{\cR,w}$ containing all the objects $\mathcal G \otimes_{\mathcal O_S} \mathcal L$ for any equivariant invertible sheaf $\mathcal L$ on $S$. By Proposition~\ref{prop:localgenerator}, this is all of $\dmodgraded{\Z^2}{\cR,w}$.
\end{proof}

\begin{lemma-section} \label{lemma: compact generators for sheafy dg-mod}
 The category, $\op{D}(\op{Mod}_{\Z} \mathcal B_w)$, is a compactly-generated triangulated category. Moreover, the set of objects $\{ \mathcal B_w \otimes_{\mathcal O_S} \mathcal L \}$ is a set of compact generators.
\end{lemma-section}

\begin{proof}
  The category, $\op{D}(\op{Mod}_{\Z} \mathcal B_w)$, admits all small coproducts so we only need to find a set of compact objects whose right orthogonal is zero. It is clear that each $\mathcal B_w \otimes_{\mathcal O_S} \mathcal L$ is compact. The proof of the generation statement is identical to the last part of the proof of Proposition \ref{prop:localgenerator}. For a $\cB_w$-module $\cM$ such that $\mathcal H^n(\mathcal M) \not = 0$, there is an equivariant invertible sheaf $\cL$ such that there is a section $\mu\in\Gamma(S, \op{ker} d_{\mathcal M}^n \otimes_{\mathcal O_S} \mathcal L^{-1})^{\mathbb{G}_m}$. This gives a non-zero morphism $\mathcal B_w \to \mathcal M \otimes_{\mathcal O_S} \mathcal L^{-1}[n]$ by sending $1$ to $\mu$ and hence a non-zero morphism $\cB_w\otimes_{\cO_S}\cL \rightarrow \cM[n]$. Thus, $\{ \mathcal B_w \otimes_{\mathcal O_S} \mathcal L\}^{\perp} = 0$. 
%
% Let us demonstrate the validity of the following claim: if $\mathcal M$ is a $\mathcal B_w$-module with $\mathcal H^{\bullet} \mathcal M \not = 0$, then there exists an equivariant invertible sheaf, $\mathcal L$, an $n \in \Z$, and a nonzero morphism $\mathcal B \otimes_{\mathcal O_S} \mathcal L \to \mathcal M[n]$. This shows that $\{ \mathcal B_w \otimes_{\mathcal O_S} \mathcal L[n] \}^{\perp} = 0$ and will finish the proof.
% 
% Let $n$ be such that $\mathcal H^n(\mathcal M) \not = 0$. Since $S$ is smooth, by \cite[Lemma 2.10]{Tho2}, there is an invertible equivariant sheaf $\mathcal L^{-1}$ and a non-zero invariant section $\mu \in \Gamma(S, \op{ker} d_{\mathcal M}^n \otimes_{\mathcal O_S} \mathcal L^{-1})^{\mathbb{G}_m}$ whose image in $\Gamma(S, \mathcal H^n \mathcal M \otimes_{\mathcal O_S} \mathcal L^{-1})^{\mathbb{G}_m}$ is also nonzero. Define a $\Z$-graded $\cB_w$-dg morphism, $\mathcal B_w \to \mathcal M \otimes_{\mathcal O_S} \mathcal L^{-1}[n]$ by sending $1$ to $\mu$ and extending over the algebra structure of $\mathcal B$. This is the desired morphism. 
\end{proof}

\begin{definition}
  A $\mathcal B_w$-module $\mathcal M$ is \newterm{perfect} if it lies in the smallest thick subcategory classically generated by the $\mathcal B_w\otimes_{\cO_S}\cL$.  The category $\op{D}_{\op{pe}}(\op{Mod}_{\Z} \mathcal B_w)$ is the full subcategory of perfect objects in $\op{D}(\op{Mod}_{\Z} \mathcal B_w)$.
% A $\mathcal B_w$-module $\mathcal M$ is \newterm{perfect} if, locally on $S$, $\mathcal M$ lies in the smallest thick subcategory generated by the $\mathcal B_w(i)$ for $i \in \Z$.  
\end{definition}

\begin{lemma-section} \label{lemma: compact = perf = B x L}
 The compact objects of $\op{D}(\op{Mod}_{\Z} \mathcal B_w)$ and the perfect objects coincide.
\end{lemma-section}

\begin{proof}
  Combining Lemma~\ref{lemma: compact generators for sheafy dg-mod}, \cite[Lemma 1.7]{Nee} and \cite[Theorem 2.1]{Nee} proves that the subcategory of all compact objects lies in the smallest thick subcategory containing all objects of the form $\mathcal B_w \otimes_{\mathcal O_S} \mathcal L$ for invertible equivariant sheaves $\mathcal L$. The other inclusion is by definition. 
\end{proof}
Before we can prove the next proposition, we need to make a quick digression. First recall the following definition.

\begin{definition}
 An additive category, $\mathcal A$, is \newterm{idempotent complete} if for every morphism $e: A \to A$ with $e^2 = e$ there is a splitting
 \begin{displaymath}
  A \cong \op{ker}(e) \oplus \op{im}(e).
 \end{displaymath}
\end{definition}

We will need a few facts about idempotent completeness.

The \newterm{idempotent completion} of $A$, denoted $\widetilde{A}$, is the additive category whose objects are pairs $(a,e)$ with $A \in \mathcal A$ and $e: A \to A$ idempotent. A morphism between $(A,e)$ and $(A',e')$ is a morphism $f: A \to A'$ in $\mathcal A$ such that $fe = e'f = f$. 
\begin{theorem-section} \label{thm: idempotent compl of triang is triang}
 Let $\mathcal T$ be a triangulated category. The idempotent completion, $\widetilde{\mathcal T}$, can be equipped uniquely with the structure of a triangulated category such that the natural inclusion $\mathcal T \to \widetilde{\mathcal T}$ is exact. Moreover, triangles in $\widetilde{\mathcal T}$ are exactly retracts of triangles in $\mathcal T$.
\end{theorem-section}

\begin{proof}
 This is \cite[Theorem 1.5, Theorem 1.12]{BS01}. 
\end{proof}
 
\begin{lemma-section} \label{lem:cutlemma}
 Let $\mathcal T = \langle \mathcal A, \mathcal B \rangle$ be a weak semi-orthogonal decomposition. Then, there is a weak semi-orthogonal decomposition $\widetilde{\mathcal T} = \langle \widetilde{\mathcal A},\widetilde{\mathcal B} \rangle$. Moreover, $\mathcal T$ is idempotent complete if and only if $\mathcal A$ and $\mathcal B$ are idempotent complete.
\end{lemma-section}
 
\begin{proof}
 The second statement is an immediate corollary of the first statement. So, assume that we have weak semi-orthogonal decomposition $\mathcal T = \langle \mathcal A, \mathcal B \rangle$. We have natural inclusions $\widetilde{\mathcal A}, \widetilde{\mathcal B} \to \widetilde{\mathcal T}$. It is clear from the definition of the idempotent completion that $\widetilde{\mathcal A} \subseteq \widetilde{\mathcal B}^{\perp}$. Let $(T,e)$ be an object of $\widetilde{\mathcal T}$. Then, $T$ sits in a triangle
 \begin{displaymath}
  B \overset{\phi}{\to} T \overset{\psi}{\to} A \overset{\lambda}{\to} B[1].
 \end{displaymath}
 There is a unique morphism of triangles 
 \begin{center}
 \begin{tikzpicture}
   \matrix (m) [matrix of math nodes, row sep=2em, column sep=3em]
   { B & T & A & B[1] \\
     B & T & A & B[1] \\};
   \path[->]
     (m-1-1) edge node[above] {$\phi$} (m-1-2)
     (m-1-2) edge node[above] {$\psi$} (m-1-3)
     (m-1-3) edge node[above] {$\lambda$} (m-1-4)
     (m-2-1) edge node[above] {$\phi$} (m-2-2)
     (m-2-2) edge node[above] {$\psi$} (m-2-3)
     (m-2-3) edge node[above] {$\lambda$} (m-2-4)
     (m-1-1) edge node[left] {$e_b$} (m-2-1)
     (m-1-2) edge node[left] {$e$} (m-2-2)
     (m-1-3) edge node[left] {$e_a$} (m-2-3)
     (m-1-4) edge node[left] {$e_b[1]$} (m-2-4)
     ;
 \end{tikzpicture}
 \end{center}
 By Lemma 2.4 in \cite{Kuz09}, which says that the diagram decomposing any object of a derived category with respect to a semi-orthogonal decomposition is unique and functorial, $e_a$ and $e_b$ are idempotents. Thus, the sequence
 \begin{displaymath}
  (B,e_b) \overset{\phi e_b}{\to} (T,e) \overset{\psi e}{\to} (A,e_a) \overset{\lambda e_a}{\to} (B[1],e_b[1])
 \end{displaymath}
 is a retract of the exact triangle
 \begin{displaymath}
  B \to T \to A \to B[1].
 \end{displaymath}
 and is, by definition, an exact triangle. Thus, we satisfy the conditions of a weak semi-orthogonal decomposition.
\end{proof}

\begin{proposition-section} \label{proposition: essentially surject onto perfects}
 The essential image of $F$ is the subcategory of perfect modules. 
\end{proposition-section}

\begin{proof}
  Since by Proposition~\ref{prop:localgenerator} $\dmodgraded{\Z^2}{\cR,w}$ is generated by objects of the form $\mathcal G \otimes_{\O_S} \mathcal L$ and $F$ is fully-faithful by Proposition~\ref{prop: ainftyequiv}, the essential image of $F$ is dense in the smallest thick subcategory containing the set of objects $\{ \mathcal B_w \otimes_{\mathcal O_S} \mathcal L \}$ for $\mathcal L$ invertible. By Lemma~\ref{lemma: compact = perf = B x L}, this is the subcategory of perfect modules.
  Finally, by Corollary \ref{proposition: VGIT for degree d fibrations} and Lemma \ref{lem:cutlemma}, $\dmodgraded{\Z^2}{\cR,w}$ is idempotent complete thus the essential image is also thick. 
\end{proof}

We can thus prove the following: 
\begin{proposition-section}\label{prop:generatoroutcome}
 There is an equivalence
 \begin{equation*}
   \dcoh{[Q_-/\mathbb{G}_m^2],w} \iso \op{D}_{\op{pe}}(\op{Mod}_{\Z} \mathcal B_w) 
 \end{equation*}
\end{proposition-section}

\begin{proof}
  Note first that by \eqref{eqn: affine} we have an equivalence of $\dcoh{[Q_-/\mathbb{G}_m^2],w}$ with $\dmodgraded{\Z^2}{\cR,w}$. Propositions \ref{prop: ainftyequiv} and \ref{proposition: essentially surject onto perfects} show that the functor $F$ gives the equivalence $$ F: \dmodgraded{\Z^2}{\cR,w} \to \op{D}_{\op{pe}}(\op{Mod}_{\Z} \mathcal B_w), $$and this completes the proof.
\end{proof}

We now calculate the endomorphism sheaf of dg-algebras of $\mathcal G$ and obtain a more explicit description of $\cB_w$. To this end, we will replace $G$ with a quasi-isomorphic locally-free factorization in $\dmodgraded{\Z^2}{\cR,w}$. 

Define a factorization by 
\begin{align*}
 \mathcal F^{-1} & := \bigoplus_{r=2s+1} \Lambda^{r}(\mathcal E)\otimes_{\cO_S}\cR (r,0) \\
 \mathcal F^0 & :=  \bigoplus_{r=2s} \Lambda^r(\mathcal E)\otimes_{\cO_S}\cR (r,0) \\
\end{align*}
with differential $$\delta_{\cF}=\op{d}_{\text{Koszul}}+\gamma\wedge\cbullet.$$

The Koszul differential $\op{d}_\text{Koszul}$ is given by the composition
\begin{align*}\Lambda^i{\cE}\otimes_{\cO_S}\op{Sym}\cE\otimes_{\cO_S}\op{Sym}(\cU,\cU^{-1})(i,0)& \rightarrow\Lambda^i{\cE}\otimes_{\cO_S}\cE^*\otimes_{\cO_S}\cE\otimes_{\cO_S}\op{Sym}\cE\otimes_{\cO_S}\op{Sym}(\cU,\cU^{-1})(i,0) \\& \rightarrow\Lambda^{i-1}{\cE}\otimes_{\cO_S}\op{Sym}\cE\otimes_{\cO_S}\op{Sym}(\cU,\cU^{-1})(i-1,0)
\end{align*} 
where the first map is induced by the map $\cO_S\rightarrow \cE^*\otimes_{\cO_S}\cE$ corresponding to the identity element in ${\mathcal End}_{\cO_S}(\cE)$, while the second one is induced by contraction $\Lambda^i{\cE}\otimes_{\cO_S}\cE^*\rightarrow\Lambda^{i-1}{\cE}$ and multiplication in $\op{Sym}\cE$.

The relative $1$-form $\gamma$ is defined as 
$$\gamma :=\frac1d\text{d}w\in\Lambda^1\cE\otimes_{\cO_S}\op{Sym}^{d-1}\cE\otimes_{\cO_S}\op{Sym}^{1}\cU$$ 
where $\text{d}w$ is the relative algebraic deRham differential of $w$.

In local coordinates, where the $x_i$ are a basis of $\cE$, $\displaystyle{\frac{\partial}{\partial x_i}}$ is the corresponding basis of $1$-forms in $\Lambda(\mathcal E^*)$, and $dx_i$ is the corresponding basis of $1$-forms in $\Lambda(\mathcal E)$, we have $\delta_{\cF}=i_\eta +\gamma\wedge\cbullet$, where $$\eta=\sum x_i\frac{\partial}{\partial x_i}\t{ and }\gamma= \frac{u}{d} \sum \frac{\partial w}{\partial x_i}dx_i.$$

\begin{lemma-section} \label{lemma: loc free replacement}
 The factorization $\mathcal F$ is quasi-isomorphic to the factorization \\
 $\mathcal G=(0,\op{Sym}(\cU,\cU^{-1}),0,0)$.
\end{lemma-section}
\begin{proof}
 This is a direct consequence of \cite[Theorem 3.9]{BDFIK}. 
\end{proof}

Since the equivalences above respect the grading shifts, we can now use the graded endomorphism algebra $\bigoplus_{i\in\Z}{\mathcal Hom}_{\mathcal R,w,\Z^2}(\mathcal F(i,0),\mathcal F)$, which we still denote by $\mathcal B_w$, to calculate $\bigoplus_{i\in\Z}{\mathbf R\mathcal Hom}_{\mathcal R,w,\Z^2}(\mathcal G(i,0),\mathcal G)$.

Note that as an $\cR:= \op{Sym}\cE \otimes_{\cO_S} \op{Sym}(\cU, \cU^{-1})$-module, we can write $\cB_w$ as 
$$\cB_w = \Lambda^{\bullet}\cE\otimes_{\cO_S}\op{Sym}\cE\otimes_{\cO_S}\op{Sym}(\cU,\cU^{-1})\otimes_{\cO_S}\Lambda^{\cbullet}\cE^*.$$ 
Under this description, a basic local section $(\beta \otimes f \otimes \theta)$, denoted briefly as $(\beta,f,\theta)$, for $\beta \in \Lambda^\bullet (\cE)$, $f\in \op{Sym}(\cE)\otimes_{\cO_S} \op{Sym}(\cU,\cU^{-1})$ and $\theta\in\Lambda^\bullet (\cE^*)$, corresponds to the endomorphism of $\cF$ that acts on basic local sections of $\cF$ by
\[
(\beta,f,\theta)(\beta',f')=(\langle\theta,\beta'\rangle\beta,ff'),
\]
where the pairing $\langle\theta,\beta'\rangle$ is the natural pairing between $\Lambda^\bullet(\cE^*)$ and $\Lambda^\bullet(\cE)$ (in particular, the pairing is $0$ unless $\theta$ and $\beta'$ live in the same wedge power, so this pairing is different from the contraction pairing).

The sections of the sheaf $\cB_w$ have a dg algebra structure with differential $\overline{\partial}$ induced by $\delta_\cF$ and product structure, induced by composition of the endomorphisms, given by
\begin{equation} \label{eq: product formula}
 m(b,b') :=: bb' := (\langle\theta,\beta'\rangle\beta,ff',\theta'),
 \end{equation}
for 
$b=(\beta,f,\theta)$ and $b'=(\beta',f',\theta')$
basic local sections as above. 

Finally, note that there are two gradings on $\cB_w$, the internal $\Z$-grading and the cohomological $\Z$-grading. Sections of $\Lambda^1 \cE$ have internal degree $-1$ and cohomological degree $-1$, sections of $\Lambda^1 \cE^*$ have the opposite gradings; whereas sections of $\cU$ have internal degree $d$ and cohomological degree $2$ and sections of $\op{Sym} \cE$ have internal degree $-1$ and cohomological degree $0$.

\section{Transferring to an \texorpdfstring{$A_{\infty}$-}{A-infinity }structure}\label{sec:ainfinitymain}
We will be working with sheaves of $A_\infty$-algebras over $S$ and modules over them. These will have an internal $\Z$-grading in addition to the usual cohomological $\Z$-grading on such objects. We require that the restriction maps for these sheaves are always strict morphisms of $A_\infty$-algebras and modules. 

In this section we will prove the following theorem:
\begin{theorem-section} \label{thm:ainfinitymain}
 There exists a sheaf of graded $A_\infty$ $\op{Sym}(\cU,\cU^{-1})$-algebras $(\cA,\mu^{\bullet})$ with 
 $$\cA=\op{Sym}(\cU,\cU^{-1})\otimes_{\cO_S}\Lambda^{\cbullet}\cE^*,$$
 such that there is an $A_\infty$-quasi-isomorphism
 \begin{equation*}
	\cA \to \cB_w.
 \end{equation*}
 
 Moreover,
 \begin{enumerate}
  \item[(i)] If $d=2$, $\cA$ is a sheaf of Clifford algebras. The Clifford relations are given in local coordinates, with $x_i$ a local basis of $\cE$ and $v_i$ the dual basis, by
$$ \mu^2(1\otimes v_i,1\otimes v_j) + \mu^2(1\otimes v_j, 1\otimes v_i) = \frac{\partial^2w}{\partial x_i\partial x_j} $$
  \item[(ii)] If $d>2$, $\cA$ is a minimal $A_\infty$ algebra with the following properties:
  \begin{itemize}
    \item The multiplication $\mu^2$ is given by the usual wedge product on $\cA$ induced from $\Lambda^\bullet(\cE^*)$.
    \item $\mu^k = 0$ for $3\leq k\leq d-1$ and, in local coordinates, we have
    $$\mu^d(1\otimes v_{i_1},\ldots,1\otimes v_{i_d})=\frac{1}{d!}\frac{\partial^d w}{\partial x_{i_1}\ldots\partial x_{i_d}}$$ 
    where the $v_{i_j}$ are not necessarily distinct.
  \end{itemize}
 \end{enumerate} 
\end{theorem-section}

\begin{remark}
Local calculations have been provided in the above statement as they are more explicit and easier to state. We will also give formulas for the global $\mu^k$ later in Lemma \ref{lem: multiplicative structure} and Proposition \ref{prop:higherproducts} below. 
\end{remark}

To prove Theorem \ref{thm:ainfinitymain}, we first observe that $\cA$ is the cohomology sheaf of algebras of a different $\Z$-graded dg-algebra $\cB$, which is the same as $\cB_w$ except that it has a modified differential and is easily seen to be formal. The strategy is to use the homological perturbation lemma to obtain an $A_\infty$ structure on $\cA = \op{H}^{\cbullet}(\cB)$ that makes it quasi-isomorphic to $\cB_w$. 

To this end, let us consider the pair $(\cF,\op{d}_{Koszul})$ instead of $(\cF,\delta_{\cF})$ where we had $\delta_{\cF}=\op{d}_{Koszul} + \gamma\wedge \cbullet$. Let $\cB$ be the endomorphism dg algebra of $(\cF,\op{d}_{Koszul})$. As we had for $\cB_w$, we have
\[
\cB=\Lambda^{\cbullet}\cE\otimes_{\cO_S}\op{Sym}\cE\otimes_{\cO_S}\op{Sym}(\cU,\cU^{-1})\otimes_{\cO_S}\Lambda^{\cbullet}\cE^*,
\]
with the same product structure described in the previous section, but with a differential now induced by $\op{d}_{Koszul}$, different from the differential of $\cB_w$.

\begin{definition}
 We denote the differential induced by $\op{d}_{Koszul}$ by $\partial:\cB\to \cB$ and differential of $\cB_w$ by $\overline\partial$ in keeping with standing notation for homological perturbation.
\end{definition}

\begin{lemma-section} \label{lemma: formality with Koszul differential}
 The sheaf of dg-algebras, $(\cB,\op{d}_{Koszul})$, is formal with 
 \begin{displaymath}
  \cA :=\op{H}^{\cbullet}(\cB) \cong \op{Sym}(\cU,\cU^{-1})\otimes_{\cO_S}\Lambda^{\cbullet}\cE^*.
 \end{displaymath}
\end{lemma-section}

\begin{proof}
 The factorization, $\mathcal F$, after forgetting the differential $\gamma \wedge \bullet$ becomes a chain complex quasi-isomorphic to $\op{Sym}(\cU,\cU^{-1})$. Since $\mathcal F$ is locally-free, we have a quasi-isomorphism
 \begin{displaymath}
  \mathcal B = \bigoplus_i \mathcal Hom_{\mathcal R,w\Z^2} (\mathcal F(i,0),\mathcal F) \cong \bigoplus_i \mathcal Hom_{\mathcal R,w\Z^2} (\mathcal F(i,0),\op{Sym}(\cU,\cU^{-1})).
 \end{displaymath}
 The latter chain complex is formal with cohomology exactly as claimed.
\end{proof}

Note that sections of $\cU$ have internal degree $d$ and cohomological degree $2$ whereas sections of $\Lambda^1 \cE^*$ have internal degree $1$ and cohomological degree $1$.

We now define the maps which will allow us to transfer the dg-structure on $\cB$ to the $A_\infty$-structure on $\A$.

We define $p:\cB\rightarrow\cA$ to be the projection by the ideal generated by $\op{Sym}\cE$ and $\Lambda^{\cbullet}\cE$. Note that this is not a map of sheaves of algebras, but a map of sheaves of chain complexes.

Next, we will define a map $i:\cA\rightarrow\cB$. In order to do that, for each $a\in\Z_{\geq0}$, consider the composition $\alpha_a$ of the maps
\begin{align*}
 & \Lambda^a\cE\otimes_{\cO_S}\op{Sym}\cE\otimes_{\cO_S}\op{Sym}(\cU,\cU^{-1})\otimes_{\cO_S}\Lambda^{\cbullet}\cE^*\\ 
 & \,\,\,\,\,\,\,\,\,\,\,\, \rightarrow\Lambda^a\cE\otimes_{\cO_S}\op{Sym}\cE\otimes_{\cO_S}\op{Sym}(\cU,\cU^{-1})\otimes_{\cO_S}(\Lambda^1\cE\otimes_{\cO_S}\Lambda^1\cE^*)\otimes_{\cO_S}\Lambda^{\cbullet}\cE^*\\
 & \,\,\,\,\,\,\,\,\,\,\,\, \rightarrow\Lambda^{a+1}\cE\otimes_{\cO_S}\op{Sym}\cE\otimes_{\cO_S}\op{Sym}(\cU,\cU^{-1})\otimes_{\cO_S}\Lambda^{\cbullet}\cE^*,
\end{align*}
where the first map is induced by the map $\cO_S\rightarrow\cE\otimes_{\cO_S}\cE^*$ corresponding to the identity in ${\mathcal End}_{\cO_S}(\cE)$, while the second one is induced by the wedge product. We define $i_0:\cA\rightarrow\cB$ to be the obvious inclusion and $i_k:\cA\rightarrow\cB$ by 
$$i_k=\frac{1}{k!}\alpha_{k-1}\circ\alpha_{k-2}\circ\ldots\circ\alpha_0\circ i_0.$$ 
We can now define $i$ by 
$$i=\sum_{k\geq0}i_k.$$

Lastly, we want to define a homotopy between $ip$ and $1$. We first define $h_0:\cB\rightarrow\cB$ to take $(\beta,f,\theta)$ to $(\text{d}f\wedge\beta,\theta)$ for basic local sections ($\beta,f,\theta) $ of $\cB$. Here, $\op{d}$ is the deRham differential. One can then define 
$$h_k=\frac{1}{s(s+1)\ldots (s+k)}\alpha_{k-1}\circ\alpha_{k-2}\circ\ldots\circ\alpha_0\circ h_0,$$ 
where $s= \text{deg }f+\text{deg }\beta$. The homotopy $h:\cB\rightarrow\cB$ is defined to be  $$h=\sum_{k\geq0}h_k.$$

\begin{lemma-section} \label{lemma: hom perturb prereqs}
 The morphisms, $h,i,p$, satisfy: 
 \begin{itemize}
  \item $i:\cA\ra \cB$ is an algebra homomorphism.
  \item $p i = 1$
  \item $h^2 = ph = hi = 0$
  \item $i p - 1 =  \op{d}_{Koszul}h + h \op{d}_{Koszul}$.
 \end{itemize}
\end{lemma-section}

\begin{proof}
 This is a straightforward, but tedious, computation. It is suppressed.
\end{proof}

\begin{proposition-section} \label{proposition: minimal model}
 There exists an $A_{\infty}$ structure, $\mu$, on $\mathcal A$ and a quasi-isomorphism
 \begin{displaymath}
  f: (\mathcal A, \mu) \to \mathcal B_w.
 \end{displaymath}
\end{proposition-section}

\begin{proof}
 Lemmas \ref{lemma: formality with Koszul differential} and \ref{lemma: hom perturb prereqs} guarantee that we can apply homological perturbation, as in \cite{KS}, Section 2.4 in \cite{Cra04}, or \cite{Ma04}, which provides the desired $\mu$ and $f$.
\end{proof}

The general formulas for the higher products \cite{KS} of Proposition~\ref{proposition: minimal model} on $\cA$ can be described as sums over ribbon trees with one root and d leaves such that the valency of any vertex is either 2 or 3. 

\begin{definition}\cite{KS,IgKi}
 A \newterm{ribbon tree} is a tree $T$ with a collection of vertices, a collection of semi-infinite edges $\{e_0,\ldots,e_n\}$ and a collection of finite edges such that:
 \begin{itemize}
   \item[(a)] Each semi-infinite edge is incident to a single vertex.
   \item[(b)] Each finite edge is incident to exactly two vertices.
   \item[(c)] The planar structure of $T$ has the semi-infinite edges $e_0,\ldots,e_n$ arranged in clockwise order. The edge $e_0$ is called the root and $e_1,\ldots,e_n$ are called the leaves.
   \end{itemize}
\end{definition}

Each such tree $T$ with $k$-leaves determines a term $\mu^k_T$ in the higher product $\mu^k$ on $\cA$. Orienting $T$ from the leaves to the root, we can explicitly describe the composition of maps that define $\mu^k_T$ as follows:

\begin{itemize}
\item For all incoming edges we have the map $i:\cA\ra \cB$

\item For the outgoing edges we have the map $p:\cB\ra \cA$

\item For all finite edges we have the map $h:\cB\ra \cB$ 

\item For a bivalent vertex we have the map $(\overline\partial-\partial):\cB\ra \cB$

\item For a trivalent vertex we have the multiplication on $\cB$
\end{itemize}

\begin{figure}[h!]
\centering
\begin{picture}(112,120)% width and height of the picture
  \put(30,20){\includegraphics[scale=0.75]{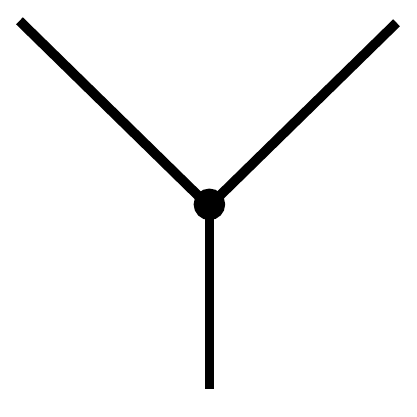}}
\put(60,60){$m$}
\put(44,100){$i$}
\put(120,100){$i$}
\put(80,24){$p$}
\put(22, 0){$\mu^2_T = p(m(i(\cbullet),i(\cbullet)))$}
\end{picture}
\hspace{0.6in}
\begin{picture}(112,120)% width and height of the picture
  \put(30,20){\includegraphics[scale=0.75]{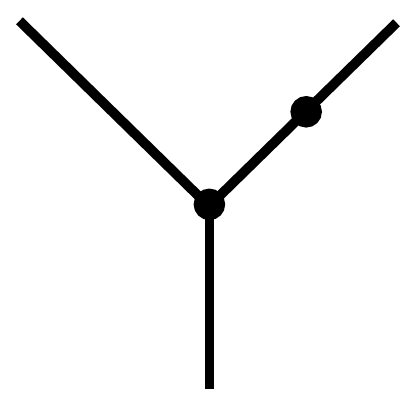}}
\put(60,60){$m$}
\put(44,100){$i$}
\put(120,100){$i$}
\put(80,24){$p$}

\put(104,79){$\overline{\partial}-\partial$}
\put(79,79){$h$}
\put(-10, 0){$\mu^2_T = p(m(i(\cbullet),(h\circ(\overline\partial - \partial)\circ i)(\cbullet))))$}

\end{picture}
\hspace{0.6in}
\caption{Examples of ribbon trees with one and two vertices.}
        \label{fig:ribbonstreees}
\end{figure}

\begin{figure}[h!]
\centering
\begin{picture}(270,240)% width and height of the picture
  \put(30,0){\includegraphics[scale=0.75]{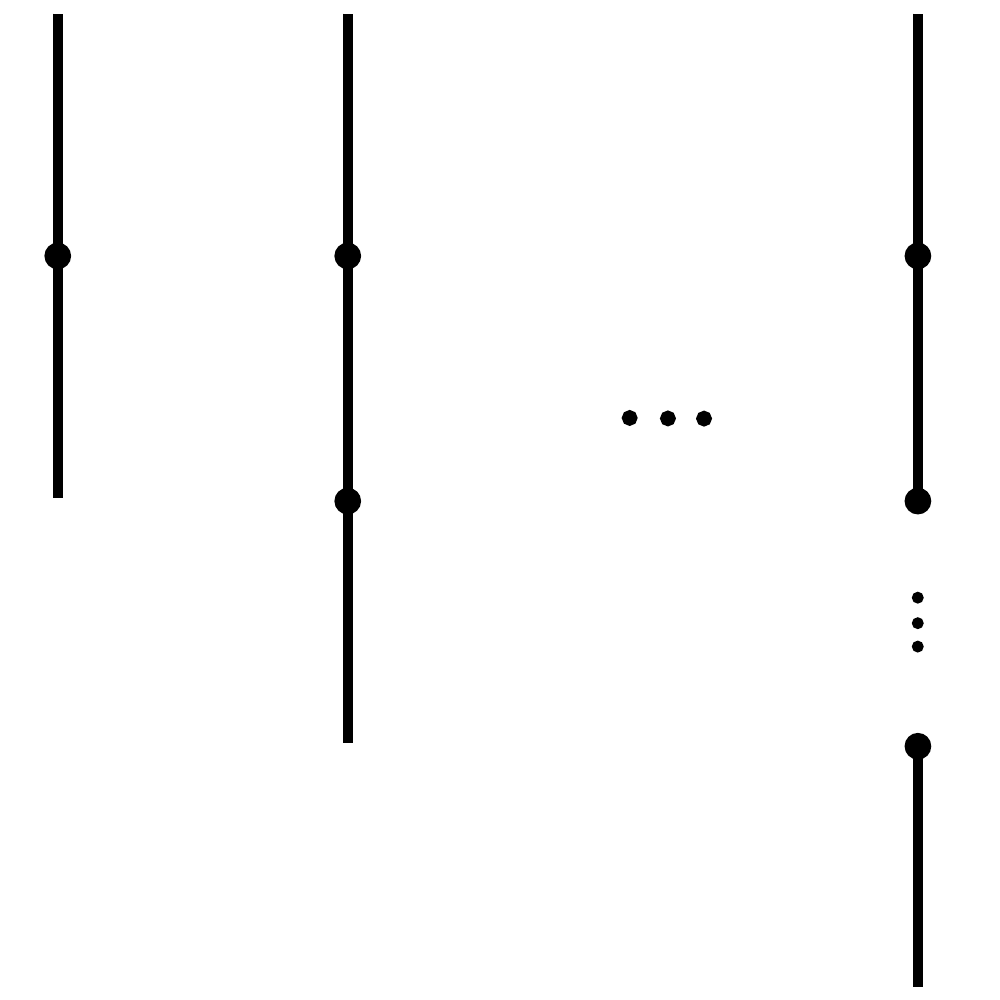}}
  \put(49,158){$\overline{\partial} - \partial$}
  \put(112,158){$\overline{\partial} - \partial$}
  \put(112,104){$\overline{\partial} - \partial$}
  \put(236,158){$\overline{\partial} - \partial$}
  \put(236,104){$\overline{\partial} - \partial$}
  \put(236,51){$\overline{\partial} - \partial$}
  
  \put(49,212){$i$}
  \put(112,212){$i$}
  \put(236,212){$i$}
  
  \put(49,104){$p$}
  \put(112,51){$p$}
  \put(236,-3){$p$}

  \put(94,131){$h$}
  \put(218,131){$h$}
\end{picture}
        \caption{Ribbon trees contributing to the differential on $\cA$}
        \label{fig:mu1}\label{fig:differential}
\end{figure}
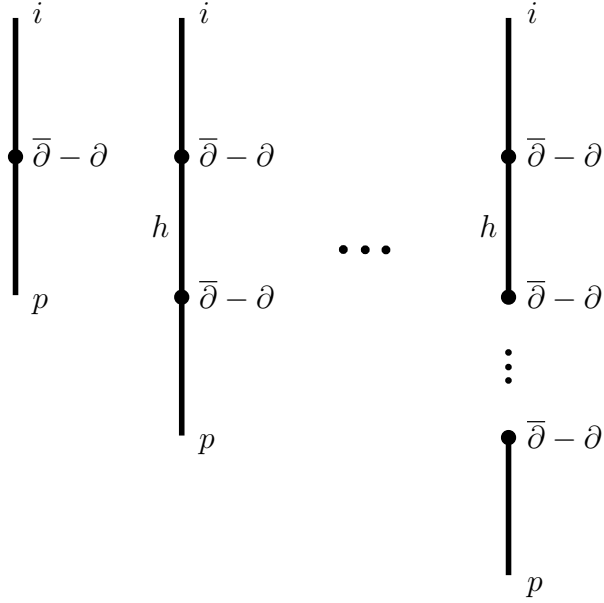

The higher products are given by
\begin{equation} \label{eq: tree summing formula}
\mu^k = \sum	 \mu^k_T,
\end{equation}
where the sum is taken over all ribbon trees with $k$-leaves.

Corollary~\ref{corollary: replacement by minimal model} above is the first part of Theorem \ref{thm:ainfinitymain}. Using the explicit description of the $\mu$ above, we will now verify the properties of the $A_\infty$-structure on $\cA$. In the process of doing so, we will make some calculations in local coordinates. More precisely, consider any affine open $U\subset S$ where $\cE$ and $\cU$ are trivial and take $\left\{ x_i \right\}$ to be a basis of $\cE|_{U}$ and $u$ in $\cU$. We will denote the corresponding basis of $\Lambda^1 \cE$ by $\left\{  dx_i\right\}$ and the dual basis of $\Lambda^1 \cE^*$ by $\left\{ \frac{\partial}{\partial x_i} \right\}$ or by $\left\{ v_i \right\}$.

In these local coordinates, the formulas above can be rewritten as follows:
\begin{itemize}
\item $i:\cA\ra \cB$ is given locally by 
\begin{equation} \label{eq: i formula}
i(r,\theta)=\sum_{k\geq0}\sum_{j_1<\ldots <j_k}(dx_{j_1}\wedge\ldots dx_{j_k},r,\frac{\partial}{\partial x_{j_k}}\wedge\ldots\wedge\frac{\partial}{\partial x_{j_1}}\wedge\theta).
\end{equation}

\item $h:\cB\ra \cB$ is given locally by \\
  \parbox[c]{\linewidth}{ 
  \begin{flalign}\label{eq: h formula}
 	   & h(\beta,f,\theta)=&\\
	   & \nonumber \sum_{k\geq0}\frac{k\!}{s(s+1)\ldots (s+k)}\sum_{j_1<\ldots <j_k}(df\wedge\beta\wedge dx_{j_1}\wedge\ldots dx_{j_k},\frac{\partial}{\partial x_{j_k}}\wedge\ldots\wedge\frac{\partial}{\partial x_{j_1}}\wedge\theta ), &
  \end{flalign}
}
 where $s= \text{deg }f+\text{deg }\beta$ (when $f\beta$ is a constant, $h$ takes the element $(\beta,f,\theta)$ to zero).
 
\item $\overline\partial-\partial$ is given locally by:
\begin{equation} \label{eq: partial formula}
(\overline\partial-\partial)((\beta,f,\theta)=(\gamma\wedge\beta,uf,\theta)+(-1)^{\text{deg }\beta}\sum_k(\beta,\frac{\partial w}{\partial x_k}uf,i_{dx_k}\theta).
\end{equation}
\end{itemize}

\begin{remark}
In the arguments that follow, we will make use of a $\Z^2$-grading on $\cB$ different from the internal and cohomological gradings we have considered so far. This is only for the purposes of the arguments below and will help us in simplifying the computation of the $A_\infty$ products on $\cA$. The two $\Z$-gradings consist of:
\begin{itemize}

\item The $f$-degree on $\cB$ is a $\Z$-grading which comes from  considering $\op{Sym(\cE)}$ with its natural grading and the other factors of the tensor product in degree zero.

\item The $\beta$-degree on $\cB$ is a $\Z$-grading which comes from  considering $\Lambda(\cE)$ with its natural grading and the other factors of the tensor product in degree zero.

\end{itemize}
\end{remark}

\begin{lemma-section} \label{lem: properties of degree}
The $f$-degree and $\beta$-degree have the following properties:
\begin{itemize}
\item The $f$-degree of $h$ is $-1$. 
\item The $f$-degree of $\overline\partial-\partial$ is $d-1$. 
\item For any $b,b' \in \cB$ homogeneous in $f$-degrees, the $f$-degree of $m(b,b')$ is the sum of the $f$-degrees of $b$ and $b'$.
\item For any $b \in \cB$ homogeneous of $\beta$-degree $s$, $h(b) \in \cB_{>s, \beta}$, where $\cB_{>s, \beta}$ denotes the set of all elements of $\beta$-degree strictly larger than $s$. 
\item For any $b \in \cB$ homogeneous of $\beta$-degree $s$, $(\overline\partial-\partial)(b) \in \cB_{\geq s, \beta}$. 
\item For any $b\in \cB$ homogeneous of $\beta$-degree $s$ and any $b'\in \cB$, $m(b,b')$ is homogeneous of $\beta$-degree $s$. 
\end{itemize}
\end{lemma-section}
\begin{proof}
This follows immediately from the definitions.
\end{proof}

\begin{remark}
The number of trees contributing to each $A_\infty$ product is finite. Indeed, any tree containing a long enough chain will not contribute to the summation because both $h$ and $\overline\partial-\partial$ increase the $\beta$-degree and all elements of positive $\beta$-degree are in the kernel of $p$.
\end{remark}

\begin{lemma-section}
The sheaf of graded $A_\infty$-algebras $\cA$ is minimal, i.e., the differential $\mu^1$ on $\cA$ is trivial.
\end{lemma-section}
\begin{proof}
  Consider the trees in Figure \ref{fig:differential}.
\noindent By the tree summing formula \eqref{eq: tree summing formula}, the differential is given by 
\[
\mu^1(a)=p(\overline\partial-\partial)i(a)+p(\overline\partial-\partial)h(\overline\partial-\partial)i(a)+p(\overline\partial-\partial)h(\overline\partial-\partial)h(\overline\partial-\partial)i(a)+\ldots. 
\]                                                  
Since $\overline\partial-\partial$ has $f$-degree $d-1$, and $p$ kills everything of positive $f$-degree we have $p (\overline\partial - \partial) = 0$ (the $f$-degree is an $\N$-grading).
\end{proof}

\begin{lemma-section}\label{lem: multiplicative structure}
The multiplicative structure $\mu^2$ on $\cA$ can be described as follows:
\begin{enumerate}
 \item[(i)] If $d>2$, $\mu^2$ is induced by the ribbon tree with two leaves and a single trivalent vertex. In particular, the multiplication is given by the usual wedge product on $\cA$ induced from the wedge product on $\Lambda^\bullet(\cE^*)$.
 \item[(ii)] If $d=2$, $\mu^2$ is induced by the ribbon tree described in (i) plus the ribbon tree with two leaves and one bivalent vertex connected to the second leaf. In particular, the multiplication satisfies the Clifford relations, given locally by 
$$ \mu^2(1\otimes v_i,1\otimes v_j) + \mu^2(1\otimes v_j, 1\otimes v_i) = \frac{\partial^2w}{\partial x_i\partial x_j},$$
and globally by
$$ \mu^2(1\otimes s_1, 1\otimes s_2) + \mu^2(1\otimes s_2, 1\otimes s_1) = 
\frac12 \op{d}(\op{d}w \, \lrcorner \, s_1) \, \lrcorner \, s_2,$$ 
where $\op{d}$ is the deRham differential and $s_1, s_2 $ are sections of $\Lambda^1 \cE^*$.
\end{enumerate}
\end{lemma-section}

\begin{proof}

\begin{figure}[h!]
\centering
\begin{picture}(112,100)% width and height of the picture
  \put(30,0){\includegraphics[scale=0.75]{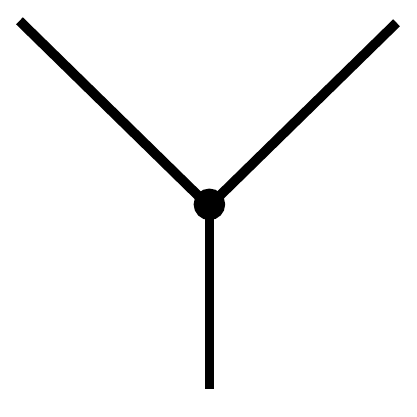}}
\put(60,40){$m$}
\put(44,80){$i$}
\put(120,80){$i$}
\put(80,4){$p$}
\end{picture}
\hspace{0.3in}
\begin{picture}(112,100)% width and height of the picture
  \put(30,0){\includegraphics[scale=0.75]{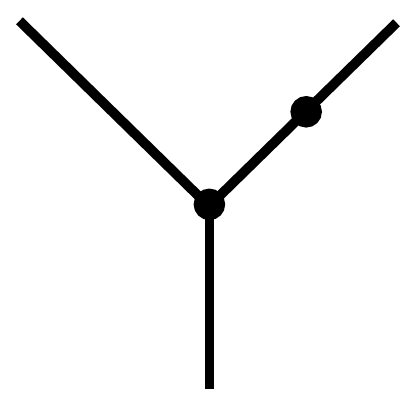}}
\put(60,40){$m$}
\put(44,80){$i$}
\put(120,80){$i$}
\put(80,4){$p$}

\put(104,59){$\overline{\partial}-\partial$}
\put(79,59){$h$}

\end{picture}
        \caption{Ribbon trees contributing to $\mu^2$}
        \label{fig:ribbonstreees2}
\end{figure}

Let $a_i=(r_i,v_i)$ and $a_j=(r_j,v_j)$ be in $\cA$ with $v_i,v_j\in \displaystyle{\left\lbrace\frac{\partial}{\partial x_1},\ldots, \frac{\partial}{\partial x_n}\right\rbrace}$.
 
\emph{(i)} Any tree $T$ contributing to the summation formula for $\mu^2$ will have exactly one trivalent vertex (since it has two leaves). Moreover, if we let $m$ be the number of bivalent vertices, we see that both $h$ and $(\overline\partial - \partial)$ appear $m$ times in $T$. However, by Lemma~\ref{lem: properties of degree}, $h$ has $f$-degree $-1$, while $\overline\partial-\partial$ has $f$-degree $d-1$, it follows that before applying $p$, the operator associated to $T$ will have $f$-degree $m(d-2)$.  Since $d>2$ the quantity $m(d-2)$ is positive if and only if $m>0$.  Consequently, if $m >0$ then the operator proceeding $p$ lies in the kernel of $p$. Therefore, the only tree contributing to the summation formula is the one with $m=0$ bivalent vertices. (c.f. Figure \ref{fig:ribbonstreees2})

Now, by what we argued above, to calculate $\mu^2(a_i,a_j)$ we need to compute $p(i(a_i)i(a_j))$. We first note that we only need to consider the $\beta$-degree 0 part of $i(a_i)$ since all the higher $\beta$-degree terms, after multiplication, will be sent to $0$ via $p$. Thus, without loss of generality, we may assume $i(a_i)=(1,r_i,v_i)$. By the definition of the product structure on $\cB$ the only part of $i(a_j)$ contributing to the product is the $\beta$-degree $1$ component (all the others will vanish when multiplied with $i(a_i)$). More precisely, we can assume $i(a_j)=(v_i^*,r_j,v_i\wedge v_j)$. Thus $a_ia_j=(r_ir_j, v_i\wedge v_j)$. 

\emph{(ii)}Again, any tree $T$ contributing to the summation formula will have exactly one trivalent vertex since it has two leaves. Moreover, if we let $m$ be the number of bivalent vertices, then $h$ and $(\overline\partial - \partial)$ appear $m$ times in $T$. 

Since $p(\overline\partial - \partial) =0$, the last vertex (the one connected to the root) must have valency $3$. We have 
\[
p(m(h(-), -)) =0,
\]
 as, by Lemma~\ref{lem: properties of degree}, $h(b) \in \cB_{>0, \beta}$ for any $b$, and $m(\cB_{>0, \beta}, \cB) \subseteq \cB_{>0, \beta}$.
It follows that the last vertex must be connected to the first leaf.  Thus,
\[
\mu^2_T(a_i,a_j) = p(m(i(a_i), P(a_j)),
\]
where $P$ is the remaining part of the operator which is attached to the ribbon tree with the edges connecting the leaf and the root to the last vertex removed and $h$ replacing $p$. Now, as we saw in part \emph{(i)}, without any loss of generality, we can assume that $i(a_i)=(1,r_i,v_i)$. Moreover, as before, the only part of $i(a_j)$ contributing to $m$ is the $\beta$-degree 1 component. Since the $\beta$-degree of $h$ is $\geq 1$ it follows that $m$ has to be 0 or 1. If $m=0$ then the tree that we obtain is the one we had in \emph{(i)}. If $m=1$ we obtain the tree with one trivalent vertex and one bivalent one connected to the second leaf. (c.f. Figure \ref{fig:ribbonstreees2})

To compute the contribution of the latter tree we first note that only the $\beta$-degree 1 part of the output of $P$ will contribute to the operator induced by $T$. Now, the operator $P$ is just $h(\overline\partial-\partial)i$ so, since $h$ already has $\beta$-degree $>0$ we can assume, with loss of generality, that $i(a_j)=(1,r_j,v_j)$ and that only the $\beta$-degree $0$ part of $(\overline\partial - \partial)$ will contribute to the operator. Last, but not least, we also see that only the $\beta$-degree 1 term in $h$ will contribute to the sum.

We now compute $P(a_j)$. We have 
\begin{align} \nonumber
   P(a_j) & = h((\overline\partial-\partial)(i(a_j))) = h(\sum_k(1,\frac{\partial w}{\partial x_k}r_j,i_{d_{x_k}}v_j)  \\ & = h(1,\frac{\partial w}{\partial x_j}r_j,1)=\frac12\sum_k(dx_k,\frac{\partial^2w}{\partial x_k\partial x_j}r_j,1). 
\end{align} 
  Therefore, the total contribution of $T$ is, 
  $$\mu^2_T(a_i,a_j) = p(m((1,r_i,v_i),\frac12\sum_k(dx_k,\frac{\partial^2w}{\partial x_k\partial x_j}r_j,1)=p(\frac12r_ir_j\frac{\partial^2w}{\partial x_i\partial x_j})=\frac12r_ir_j\frac{\partial^2w}{\partial x_i\partial x_j}.$$

Thus, we have proved that $\mu^2(a_i,a_j)=(r_ir_j,v_i\wedge v_j)+(\frac12r_ir_j)\frac{\partial^2w}{\partial x_i\partial x_j}$ and this is a Clifford multiplication on $\cA$. Similarly, we have that $\mu^2(a_j,a_i)=(r_ir_j,v_j\wedge v_i)+(\frac12r_ir_j)\frac{\partial^2w}{\partial x_i\partial x_j}$ and therefore $\mu^2(a_i,a_j)+\mu^2(a_j,a_i)=r_ir_j\frac{\partial^2w}{\partial x_i\partial x_j}$ which gives the Clifford algebra structure on $\cA$. 

The global calculation follows directly from this local version.
\end{proof}

The following proposition follows a similar argument to those in \cite{Se}:

\begin{proposition-section} \label{prop: initially trivial}
The $A_\infty$-structure on $\cA$ coming from the tree summation formula agrees with the trivial $A_\infty$-structure up to order $d-1$.
\end{proposition-section}
\begin{proof}
We want to show that any tree with $k$ leaves, for $k < d$, contributes a trivial operator to the summation and therefore $\mu^k =0$ for $k < d$.
    
Consider now the $A_\infty$ product, $\mu^k$ for $k>2$. Note that for $k > 2$, any tree with no bivalent vertex does not contribute to the summation as the term $h(i(a)i(a'))=h(i(aa'))=0$ necessarily appears as the output of a trivalent vertex and thus the operator induced by such a tree would be 0. Therefore, the number $m$ of bivalent vertices, is at least $1$. 

 By Lemma~\ref{lem: properties of degree}, $h$ has $f$-degree $-1$, while $\overline\partial-\partial$ has $f$-degree $d-1$.  Since $T$ has $k$-leaves, it has exactly $k-1$ trivalent vertices.  Therefore $h$ appears in the operator at most $k-2+m$ times while $\overline\partial-\partial$ appears $m$-times.  Since $m$ preserves the $f$-degree, it follows that before applying $p$, the operator will have $f$-degree $m(d-1) - (k-2+m)$.  Therefore if
 \[
 m(d-1) - (k-2+m) >0
 \]
 then the operator vanishes.
 
In summary, in order to get a non-trivial contribution from trees with $m$ bivalent vertices one must have
\begin{equation} \label{eq: tree inequality}
1 \leq m \leq \frac{k-2}{d-2}.
\end{equation}
Thus, $\mu^k$ can be non-trivial only when $d\leq k$.
\end{proof}

\begin{lemma-section} \label{lem: 1 bivertex}
Any tree providing a nontrivial contribution to $\mu^d$ has exactly $1$ bivalent vertex.
\end{lemma-section}
\begin{proof}
Equation~\eqref{eq: tree inequality} for $k=d$ tells us that the number of bivalent vertices must be $1$.
\end{proof}

\begin{figure}[h]\vspace{0.2in}\includegraphics[width=2in]{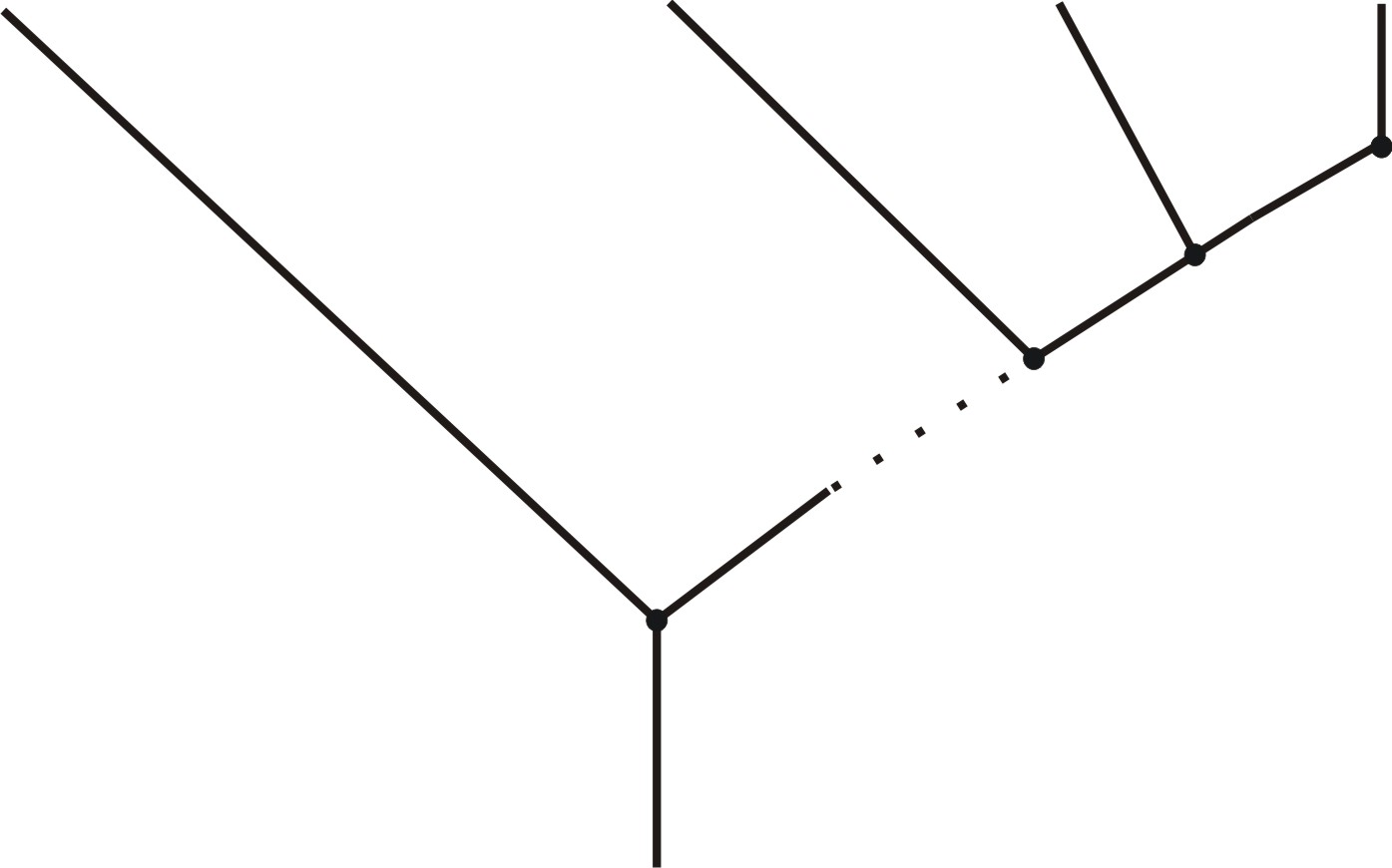}
\captionof{figure}{Solitary Tree Contribution for $\mu^d$}\label{fig: tree}
\end{figure}

The following proposition completes the proof of Theorem \ref{thm:ainfinitymain}:

\begin{proposition-section}\label{prop:higherproducts}
 For $d\geq 3$, the $d$-th higher multiplication on $\cA$ can be described locally by
 $$\mu^d(r_1 \otimes v_{i_1},\ldots,r_d\otimes v_{i_d})=\frac{1}{d!}\frac{\partial^d w}{\partial x_{i_1}\ldots\partial x_{i_d}}r_1\ldots r_d,$$ 
 where $v_{i_j}$ are not necessarily distinct elements of the basis $\left\{ v_{i}=\frac{\partial}{\partial x_i} \right\}$; and globally by
 $$ \mu^d(1\otimes s_1, \dots , 1\otimes s_d) = \frac{1}{d!} (\op{d}(\dots(\op{d}(\op{d}w \, \lrcorner \, s_1)\,\lrcorner\,s_2)\dots)\,\lrcorner\,s_d).$$
\end{proposition-section}

\begin{proof}
We consider $a_j=(r_j,v_{i_j})$ for $j=1,\ldots,d$ and  $v_{i_j}=\frac{\partial}{\partial x_{i_j}}$. Before calculating $\mu^d(a_1,\ldots,a_d)$ we first note that there is only one tree contributing to the summation formula (Figure \ref{fig: tree}).

Indeed, let $T$ be any tree that contributes to the summation formula for $\mu^d(a_1,\ldots,a_d)$. By Lemma~\ref{lem: 1 bivertex}, $T$ has only $1$ bivalent vertex.  On the other hand, since $T$ has $d$ leaves, it has exactly $d-1$ trivalent vertices. Moreover, the bivalent vertex has to be connected to one of the last two leaves since otherwise there is at least one $h$ appearing before $\overline\partial - \partial$ and that means the $f$-degree of the output of the operator given by T (before applying $p$) has $f$-degree greater than or equal to 1 and thus it lies in the kernel of p.

Therefore, the vertex connected to the root has valency $3$.  Furthermore, arguing as in Lemma~\ref{lem: multiplicative structure} it follows that this last vertex is connected to a leaf.  Thus,
\[
\mu^d_T(a_1, \ldots, a_d) = p(m(i(a_1), P(a_2, \ldots, a_d)),
\]
where $P$ is the remaining part of the operator which, as before, is attached to the ribbon tree with the edges connecting the leaf and the root to the last vertex removed and $h$ replacing $p$.  In summary, we have established that the last/left-most vertex appears as in Figure~\ref{fig: tree}. Moreover, we note that, without any loss of generality, we can assume $i(a_1)=(1,r_1,x_{i_1})$. This follows from the same argument as above, when we calculated the multiplication on $\cA$. This also forces $P(a_2, \ldots, a_d)$ to have $\beta$-degree 1 (or, more precisely, all other terms in $P(a_2, \ldots, a_d)$ will vanish when multiplied with $i(a_1)$).

We now argue that the one bivalent vertex cannot be connected to the penultimate leaf since, in this case, the operator $P$ would have $\beta$-degree $\geq 2$ and thus the operator induced by T would be 0. This follows inductively by tracing the $\beta$-degree of the operator induced by the tree we are considering. We thus conclude that there is only one tree contributing to the summation formula giving $\mu^d$ and we can calculate its contribution using a similar calculation as in Lemma~\ref{lem: multiplicative structure} which yields the desired result.

The formula for the global version follows directly from this local version.  
\end{proof}

\subsection{The derived category and a technical assumption} Recall that $\cA$ and $A_\infty$ $\cA$-modules are assumed to have strict restriction morphisms.  We start with the following definition:
\begin{definition}
  For a sheaf of graded $A_\infty$ algebras $\cA$, let $\dpemodinftyZ{\cA}$ be the smallest thick triangulated subcategory of the derived category of strictly unital graded $A_\infty$ $\cA$-modules containing all modules of the form $\cA\otimes_{\cO_S}\cL$ for graded invertible sheaves $\cL$ on $S$. If we wish to emphasize the underlying variety, we will also denote this category by $\dpemodinftyZ{(S,\cA)}$.
\end{definition}

Recall that for a sheaf of graded dg $\cO_S$-algebras $\cB$, $ \dpemodz{\cB}$ is the full subcategory of the derived category of dg $\cB$-modules, classically generated by objects of the form $\mathcal B_w\otimes_{\cO_S}\cL$. 

The following conjecture is believed to be true, however a proof is beyond the scope of this paper.
\begin{conjecture}\label{conditionstar}
	If a sheaf $\cB$ of $\Z$-graded differential-graded $\cO_S$-algebras is graded-quasi-isomorphic to a sheaf $\cA$ of $\Z$-graded $A_\infty$ $\cO_S$-algebras, then there is a $\P(S)$-linear equivalence
\begin{equation*}
	 \dpemodinftyZ{\cA} \iso \dpemodz{\cB}.
\end{equation*}
\end{conjecture}

\begin{remark}
  When $S=\op{Spec}k$, this conjecture is known by the results in \cite{lefevre-hasegawa}, Section 2.4.2; in which case, the category of $A_\infty$-modules has a model structure where the weak equivalences are quasi-isomorphisms and every object is fibrant and cofibrant. This makes considerations of derived functors more straightforward.  
For the general case, one would need to define the appropriate model structures and prove the derived-equivalences induced by the diagram in the proof of Lemma 2.4.2.3. in \emph{loc. cit.}.
\end{remark}

In the notation of Section \ref{sec:ainfinitymain}, we have the following:
\begin{corollary} \label{corollary: replacement by minimal model}
  Assuming Conjecture \ref{conditionstar}, the quasi-isomorphism, $f$, induces an equivalence,
 \begin{align*}
   \dpemodinftyZ{\mathcal A} & \cong \dpemodz{\mathcal B_w}.
 \end{align*}
 In particular, 
 \begin{displaymath}
   \dcoh{[Q_-/\GG^2],w} \cong \dpemodinftyZ{\mathcal A}.
 \end{displaymath}
\end{corollary}

\begin{proof}
  The first statement follows from Theorem \ref{thm:ainfinitymain} and Conjecture \ref{conditionstar}. The second is then an immediate consequence of Propositions~\ref{prop: ainftyequiv} and \ref{proposition: essentially surject onto perfects}.
\end{proof}

\begin{remark}
Assuming Conjecture \ref{conditionstar}, the results of Section \ref{sec:morita} and Corollary \ref{corollary: replacement by minimal model} could have alternatively been obtained (and generalized) by proving a relative version of Baranovsky's BGG correspondence for projective complete intersections \cite{Bar05}. 
\end{remark}

\section{Homological Projective Duality for d-th degree Veronese embeddings}

We will now apply the results of the previous sections to construct a homological projective dual to the degree $d$ Veronese embedding. In view of potential applications, we will do this in the relative setting. Then, if $d=2$, we will recover Kuznetsov's construction for degree two Veronese embeddings \cite{Kuz05} (when $S$ is a point) and the relative version in \cite{ABB}.

Let $S$ be a smooth, connected variety and $\mathcal P$ be a locally-free coherent sheaf on $S$. We consider the relative degree $d$ Veronese embedding for $d > 0$, 
  \begin{displaymath}
   g_d: \P_S(\mathcal P) \to \P_S(S^d\mathcal P).
  \end{displaymath}

Notice that $g_d^*(\cO_{\P({S^d\mathcal P})}(1))\cong \cO_{\P({\mathcal P})}(d)$. Consider the Lefschetz decomposition 
$$\dbcoh{\P_S(\mathcal P)}=\langle\cA_0,\ldots,\cA_i(i)\rangle$$
where the subcategories $\cA_j$ are defined to be
\begin{align*}
\cA_0=\ldots =\cA_{i-1}=\langle p^*\dbcoh{S},\ldots,p^*\dbcoh{S}(d-1)\rangle \\
\cA_i=\langle p^*\dbcoh{S},\ldots,p^*\dbcoh{S}(k-1)\rangle
\end{align*} 
where $k=\text{rk }\mathcal P-d(\lceil\frac{\text{rk }\mathcal P}{d} \rceil-1)$. In \cite[Section 4]{BDFIKhpd} we recover this Lefschetz decomposition via VGIT methods.
The universal degree $d$ polynomial $w$ is given by
\begin{displaymath}
   w := (g_d \times 1)^* \theta \in \Gamma(\P_S(\mathcal P)) \times_S \P_S(S^d\mathcal P^*), \mathcal O_{\P_S(\mathcal P)}(d) \boxtimes \mathcal O_{\P_S(S^d\mathcal P^*)}(1)),
\end{displaymath}
where $\theta$ is the tautological section in $\Gamma( \P_S(S^d\mathcal P) \times_S \P(S^d \mathcal P^*),\mathcal O_{\P(S^d\cP)}(1) \boxtimes \mathcal O_{\P_S(S^d\mathcal P^*)}(1))$. 
The zero locus $w$ in $\P_S(\mathcal P) \times_S \P_S(S^d\mathcal P^*)$ is the universal hyperplane section $\cX$ of $\P_S(\mathcal P)$ with respect to the embedding $g_d$.

The main result of \cite{BDFIKhpd} is the construction of a Landau-Ginzburg model which is a homological projective dual.

\begin{theorem-section}\label{prop:veroneselghpd}
  The gauged Landau-Ginzburg model $([\op{V}_S(\mathcal P)\times_S\P_S(S^d\mathcal P^*)/{\gm}],w)$ is a weak homological projective dual to $\P_S(\mathcal P)$ with respect to the embedding $g_d$ and the Lefschetz decomposition constructed above.

Moreover, we have:
\begin{itemize}
\item The derived category of the Landau-Ginzburg model $([\op{V}_S(\mathcal P)\times_S\P_S(S^d\mathcal P^*)/{\gm}],w)$ admits a dual Lefschetz collection
  \[ \dcoh{[\op{V}_S(\mathcal P)\times_S\P_S(S^d\mathcal P^*)/{\gm}],w} = \langle\mathcal B_{j}(-j),\ldots,\mathcal B_1(-1),\mathcal B_0\rangle \]

\item Let $\cV = S^d\mathcal P^*/\cU$ be a quotient bundle and $\cW=(S^d\mathcal P)/{\cU^\perp}$. Assume that $\P_S(\mathcal P) \times_{\P_S(S^d\mathcal P)} \P_S(\cW)$ is a complete linear section (not necessarily smooth). Then, there exist semi-orthogonal decompositions,
\[
\dbcoh{\P_S(\mathcal P) \times_{\P_S(S^d\mathcal P)} \P_S(\cW)}=\langle\mathcal C_{\cV},\mathcal A_r(1),\ldots,\mathcal A_i(i-r+1)\rangle,
\]
 and,
 \[
   \dcoh{[\op{V}_S(\mathcal P)\times_S\P_S(\cV)/{\gm}],w } =\langle\mathcal B_{0}(-r-N -2),\ldots,\mathcal B_{N -1-r}(-1),\mathcal C_{\cV}\rangle.
 \]
 \end{itemize}

\end{theorem-section}
\begin{proof}
  This is \cite[Theorem 4.1]{BDFIKhpd}.
\end{proof}
\begin{remark-section}
  For the first part of the theorem, we can alternatively consider $\cX$ as a degree $d$ hypersurface fibration over $\P(S^d\cP^*)$ and use our relative version of Orlov's theorem (Corollary \ref{proposition: VGIT for degree d fibrations}) with $S=\P_S(S^d\mathcal P^*)$, $\mathcal E=\pi^*\mathcal P$ and $\mathcal U=\O_{\P_S(S^d\mathcal P^*)}(1)$ to get the decomposition 
\begin{align*}
& \dbcoh{\mathcal X}= \\
& \,\,\,\,\, \langle\dcoh{[V_{\P_S(S^d\mathcal P^*)}(\pi^*\mathcal P)/\gm],w}, \mathcal A_1(1)\otimes\dbcoh{\P_S(S^d\mathcal P^*)},\ldots, \mathcal A_i(i)\otimes\dbcoh{\P_S(S^d\mathcal P^*)} \rangle.
\end{align*}
Observing that $V_{\P_S(S^d\mathcal P^*)}(\pi^*\mathcal P) \iso \op{V}_S(\mathcal P)\times_S\P_S(S^d\mathcal P^*)$, we obtain the required semi-orthogonal decomposition. 
\end{remark-section}
\
Using this result, we can state the following:
\begin{theorem-section}\label{thm:relveronese} 
 Let $S$ be a smooth, connected variety and $\mathcal P$ be a locally-free coherent sheaf over $S$. Let $d\geq3$ and $d\leq \op{rank}\cP$. Assuming Conjecture \ref{conditionstar}, there exists a sheaf of minimal $A_\infty$-algebras $(\mathcal A, \mu)$ on $\P_S(S^d\mathcal P^*)$ with

 $$\mathcal A=\left(\bigoplus_{k\in \Z} u^k\cO_{\P_S(S^d\mathcal P^*)}(k) \right)\otimes\Lambda^{\bullet}\mathcal P^*,$$
and
\begin{itemize}
	\item If $d>2$: $\mu^i = 0$ for $2<i<d$ and, in local coordinates,
  $$\mu^d(1\otimes v_{i_1},\ldots,1\otimes v_{i_d})=\frac{1}{d!}\frac{\partial^d w}{\partial x_{i_1}\ldots\partial x_{i_d}}.$$
  	\item If $d=2$: $\mu^i = 0$ for $i>2$ and $(\cA,\mu^2)$ is a sheaf of Clifford algebras with Clifford relations given, in local coordinates, by
$$ \mu^2(1\otimes v_i,1\otimes v_j) + \mu^2(1\otimes v_j, 1\otimes v_i) = \frac{\partial^2w}{\partial x_i\partial x_j} $$
\end{itemize}
 such that the non-commutative variety 
$$(\P_S(S^d\mathcal P^*),\mathcal A)$$ 
is a weak homological projective dual to $\P_S(\mathcal P)$ with respect to the embedding $g_d$ and the Lefschetz decomposition constructed above.
\end{theorem-section}

\begin{theorem-section}\label{thm:relveroneseapps}
  With the same assumptions as in Theorem \ref{thm:relveronese} above, we have the following
  \begin{enumerate}
   \item[(i)] The perfect derived category of the non-commutative variety $(\P_S(S^d\mathcal P^*),\mathcal A)$ admits a dual Lefschetz collection
     \[ \dpemodinftyZ{(\P_S(S^d\mathcal P^*),\mathcal A)} = \langle\mathcal B_{j}(-j),\ldots,\mathcal B_1(-1),\mathcal B_0\rangle \]
   \item[(ii)]  Let $\cV = S^d\mathcal P^*/\cU$ be a quotient bundle with $\op{rank}\cU=r$ and $\cW=(S^d\mathcal P)/{\cU^\perp}$. Assume that $\P_S(\mathcal P) \times_{\P_S(S^d\mathcal P)} \P_S(\cW)$ is a smooth, complete linear section, i.e.
\[
\op{dim}(\P_S(\mathcal P) \times_{\P_S(S^d\mathcal P)} \P_S(\cW)) = \op{dim}(\P_S(\mathcal P))-r.
\]
Then, there exist semi-orthogonal decompositions:
 \[
   \dbcoh{\P_S(\mathcal P) \times_{\P_S(S^d\mathcal P)} \P_S(\cW)}=\langle\dpemodinftyZ{(\P_S(\cV),\mathcal A|_{\P_S(\cV)})},\mathcal A_r(1),\ldots,\mathcal A_i(i-r+1)\rangle,
\] if $r\leq \frac{\op{rank}\mathcal E}{d}$
or
\[
   \dpemodinftyZ{(\P_S(\cV),\mathcal A|_{\P_S(\cV)})}=\langle\mathcal B_j(j),\ldots,\mathcal B_k(-k),\dbcoh{\P_S(\mathcal P) \times_{\P_S(S^d\mathcal P)} \P_S(\cW)}\rangle,
 \] where $k=\op{rank}\mathcal E-r-1$, if $r\geq \frac{\op{rank}\mathcal E}{d}$
  \end{enumerate}
\end{theorem-section}

\textit{Proof of Theorem \ref{thm:relveronese}} :
In the setup of Section \ref{sec:morita}, we take  $S=\P_S(\cV)$, $\cE = (\pi^*\cP)|_{\cV}$, $\cU = \cO_{\P_S(\cV)}(1)$ and the same $w$. Applying Proposition \ref{prop:generatoroutcome} gives the equivalence between the triangulated category $\dcoh{[\op{V}_S(\mathcal P)\times_S\P_S(\cV)/{\gm}],w}$ and $\dpemodinftyZ{(\P_S(\cV),\mathcal A|_{\P_S(\cV)})}$, for any quotient bundle $\cV$ of $ S^d\cP^*$. It is easy to see that this is $\P(S^d\cP^*)$-linear. Theorem \ref{thm:ainfinitymain} and the first part of \ref{prop:veroneselghpd} allows us to complete the proof.

\textit{Proof of Theorem \ref{thm:relveroneseapps}}:
This is the second part of Theorem \ref{prop:veroneselghpd} applied to the homological projective dual obtained in Theorem \ref{thm:relveronese}.

\begin{remark-section}
  As an intermediate result, the conclusions of Theorem \ref{thm:relveronese} and Theorem \ref{thm:relveroneseapps} hold when the non-commutative variety $(\P_S(S^d\mathcal P^*),\mathcal A)$ is replaced by $(\P_S(S^d\mathcal P^*), \cB_w)$ without the need for Conjecture \ref{conditionstar}. 
\end{remark-section}

\begin{corollary-section}
For any linear subspace $L\subset S^dV^*$ such that the corresponding intersection of degree $d$ hypersurfaces in $\P(V)$ is smooth and complete (i.e. it has dimension $\op{dim}(\P(V))-\op{dim}(L)$), we have
\begin{itemize}
  \item If $\op{dim}(L)\leq n/d$, then there exists a semi-orthogonal decomposition:
    $$\dbcoh{X_L}=\langle\dpemodinftyZ{(\P(L),\mathcal A|_{L})},\cA_r(r),\ldots,\cA_i(i)\rangle.$$
\item If $\op{dim}(L)\geq n/d$, then
there exists a semi-orthogonal decomposition:
 $$\dpemodinftyZ{(\P(L),\mathcal A|_{L})}=\langle\mathcal B_j(-j),\ldots,\cB_{\op{dim}(V)-\op{dim}(L)-1}(-\op{dim}(V)+\op{dim}(L)+1), \dbcoh{X_L}\rangle$$
\end{itemize}
\label{cor: d=2}
\end{corollary-section}
\begin{proof}
This follows from setting $S$ equal to $\op{Spec } k$ in Theorem~\ref{thm:relveroneseapps}.
\end{proof}

\begin{remark-section}
  If $d=2$ then, as we have noticed in the previous section, $\cA$ is actually a sheaf of Clifford algebras, therefore by \cite{Riche}, we get an equivalence
 \[
   \op{D}_{\op{pe}}(\op{Mod}_{\Z}(\P(S^2V^*),\mathcal \cB_w))\cong \op{D}_{\op{pe}}(\op{Mod}_{\Z}(\P(S^dV^*),\mathcal \cA)),
 \]
without the need for Conjecture \ref{conditionstar}. Now, using Proposition 3.7 in \cite{Kuz05}, one has 
 \[
 \op{D}_{\op{pe}}(\op{Mod}_{\Z}(\P(S^2V^*),\mathcal \cA))\cong \op{D}^{\op{b}}(\op{mod}(\P(S^dV^*),\mathcal B_0)),
 \]
 where $\mathcal B_0$ is the sheaf of even Clifford algebras defined in \cite{Kuz05}. This recovers the homological projective dual in \cite{Kuz05}. Similarly, for a smooth complete intersection of quadrics, we get the same description as in \emph{loc.\ cit.} using Corollary~\ref{cor: d=2}.  The relative versions in \cite{Kuz05} and  \cite{ABB} follow similarly.
\end{remark-section}


\begin{thebibliography}{99}

\bibitem[ABB11]{ABB}{}
A. Auel, M. Bernardara, M. Bolognesi. \emph{Fibrations in complete intersections of quadrics, Clifford algebras, derived categories, and rationality problems}. \href{http://arxiv.org/pdf/1109.6938.pdf}{arXiv:1109.6938}.

\bibitem[ADS12]{ADS}
N. Addington, W. Donovon, E. Segal, \emph{The Pfaffian-Grassmannian correspondence revisited} 
Video lecture available at: \href{http://www.newton.ac.uk/programmes/MOS/seminars/040611301.html}{www.newton.ac.uk/programmes/MOS/seminars/040611301.html} (preprint)

\bibitem[BDFIK13]{BDFIKhpd}{}
M. Ballard D. Deliu, D. Favero, U. Isik, L. Katzarkov \emph{Homological Projective Duality via Variation of Geometric Invariant Theory Quotients}. \href{http://arxiv.org/abs/1306.3957}{arXiv:1306.3957}. 


\bibitem[BDFIK12]{BDFIK}
M. Ballard, D. Deliu, D. Favero, M. U. Isik, L. Katzarkov. {\em Resolutions in factorization categories}. \href{http://arxiv.org/pdf/1212.3264}{arXiv:1212.3264}. Updates to appear.
  
\bibitem[BFK11]{BFK11}{}
M. Ballard, D. Favero, L. Katzarkov. \emph{A category of kernels for graded matrix factorizations and its implications for Hodge theory}. \href{http:arxiv.org/pdf/1105.3177v4}{arXiv:1105.3177}. 

\bibitem[BFK12]{BFK12}{}
M. Ballard, D. Favero, L. Katzarkov. \emph{Variation of Geometric Invariant Theory quotients and derived categories}. \href{http://arxiv.org/pdf/1203.6643}{arXiv:1203.6643}. 

\bibitem[BS01]{BS01}
P. Balmer, M. Schlichting. {\em Idempotent completion of triangulated categories}. J. Algebra 236 (2001), no. 2, 819-834.
%
\bibitem[Bar05]{Bar05}
V. Baranovsky. {\em BGG correspondence for projective complete intersections}. Int. Math. Res. Not. 2005, no. 45, 2759-2774.
%
% \bibitem[BIK08]{BIK}
% D. Benson, S.B. Iyengar, H. Krause. {\em Local cohomology and support for triangulated categories}. Ann. Sci. \'Ec. Norm. Sup\'er. (4) 41 (2008), no. 4, 573-619.

\bibitem[B-ZNP13]{BZNP}
D. Ben-Zvi, D. Nadler, A. Preygel. {\em Morita theory of the affine Hecke category and coherent sheaves on the commuting stack}. In preparation.

% \bibitem[BL94]{BerLun}
% J. Bernstein, V. Lunts. {\em Equivariant sheaves and functors}. Lecture Notes in Mathematics, 1578. Springer-Verlag, Berlin, 1994.
%
% \bibitem[Bei78]{Bei}{}
% A. Beilinson. \emph{Coherent sheaves on $\mathbb{P}^n$ and problems in linear algebra}. Functional Anal. Appl. 12 (1978), no.3, 214-216.

% \bibitem[B-B73]{BB}{}
% A. Bia\l{}ynicki-Birula. \emph{Some theorems on actions of algebraic groups}. Ann. of Math. (2) 98 (1973), 480-497.

\bibitem[BC09]{BC}{}
L. Borisov, A. C\u{a}ld\u{a}raru. \emph{The Pfaffian-Grassmannian derived equivalence}, J. Algebraic Geom. 18 (2009), no. 2, 201-222.

\bibitem[Bon89]{Bon}{}
A. Bondal. \emph{Representations of associative algebras and coherent sheaves}. (Russian) Izv. Akad. Nauk SSSR Ser. Mat. 53 (1989), no. 1, 25-44; translation in Math. USSR-Izv. 34 (1990), no. 1, 23-42.

\bibitem[BK90]{BK}{}
A. Bondal, M. Kapranov. {\em Representable functors, Serre functors, and reconstructions}. (Russian) Izv. Akad. Nauk SSSR Ser. Mat. 53(6), 1183-1205, 1337 (1989); translation in Math. USSR-Izv. 35(3), 519-541 (1990).

\bibitem[BO95]{BO95}{}
A. Bondal, D. Orlov. {\em Semi-orthogonal decompositions for algebraic varieties.} Preprint MPI/95-15. \href{http://arxiv.org/pdf/math.AG/9506012.pdf}{arXiv:math.AG/9506012}.

\bibitem[BO02]{BOICM} 
A. Bondal, D. Orlov. {\em Derived categories of coherent sheaves}. Proceedings of the International Congress of Mathematicians, Vol. II (Beijing, 2002), 47-56, Higher Ed. Press, Beijing, 2002.

\bibitem[BV03]{BvdB}
A. Bondal, M. Van den Bergh. {\em Generators and representability of functors in commutative and noncommutative geometry}. Mosc. Math. J. 3 (2003), no. 1, 1-36, 258.

% \bibitem[BrP90]{BP}{}
% M. Brion, C. Procesi. \emph{Action d'un tore dans une vari\'{e}t\'{e} projective}. in Operator algebras, unitary representations, enveloping algebras, and invariant theory (Paris, 1989), 509-539, Progr. Math., 92, Birkh\"auser Boston, Boston, MA, 1990.

\bibitem[Bri02]{bridgelandflops}
T. Bridgeland. {\em Flops and derived categories}. Invent. Math. 147 (2002), no. 3, 613-632.

\bibitem[Buc86]{Buc86}
R.-O. Buchweitz. {\em Maximal Cohen-Macaulay modules and Tate-Cohomology over Gorenstein rings}. Preprint, 1986.
%\newblock available at \url{http://hdl.handle.net/1807/16682}.

\bibitem[CDHPS10]{CDHPS}
A. C\u{a}ld\u{a}raru, J. Distler, S. Hellerman, T. Pantev, E. Sharpe. {\em Non-birational twisted derived equivalences in abelian GLSMs}. Comm. Math. Phys. 294 (2010), no. 3, 605-645.

\bibitem[Cra04]{Cra04}
M. Crainic. {\em On the perturbation lemma, and deformations}. \href{http://arxiv.org/pdf/math/0403266v1}{arXiv:0403.266}.

\bibitem[Del11]{deliu}
Dragos Deliu.
\newblock {Homological Projective Duality for Gr(3,6)}.
\newblock Dissertation for the degree of Doctor in Philosophy at the University
  of Pennsylvania. Available at
  \href{http://repository.upenn.edu/dissertations/AAI3463052}{http://repository.upenn.edu/dissertations/AAI3463052}, 2011.

\bibitem[DH98]{DH98}{}
I. Dolgachev, Y. Hu. \emph{Variation of Geometric Invariant Theory quotients}. With an appendix by N. Ressayre. Inst. Hautes \'{E}tudes Sci. Publ. Math. No. 87 (1998), 5-56. 

\bibitem[DSh08]{DS}
R. Donagi, E. Sharpe. {\em GLSMs for partial flag manifolds}. J. Geom. Phys. 58 (2008), no. 12, 1662-1692.

\bibitem[DSe12]{DSe}
W. Donovan, E. Segal. {\em Window shifts, flop equivalences and Grassmannian twists}. \href{http://arxiv.org/pdf/1206.0219.pdf}{arXiv:1206.0219}.

% \bibitem[Dri04]{Drinfeld}
% V. Drinfeld. {\em DG quotients of DG categories}. J. Algebra 272 (2004), no. 2, 643-691.

\bibitem[Dyc11]{Dyc11} 
T. Dyckerhoff. {\em Compact generators in categories of matrix factorizations}. Duke Math. J. 159 (2011), no. 2, 223-274.

\bibitem[Efi12]{Ef}
A. Efimov. {\em Homological mirror symmetry for curves of higher genus}. Adv. Math. 230 (2012), no. 2, 493-530.

\bibitem[Eis80]{EisMF}{}
D. Eisenbud. \emph{Homological algebra on a complete intersection, with an application to group representations}. Trans. Amer. Math. Soc. 260 (1980), no. 1, 35-64. 

% \bibitem[FHT01]{FHT}
% Y. F\'elix, S. Halperin, J.-C. Thomas. Rational homotopy theory. Graduate Texts in Mathematics, 205. Springer-Verlag, New York, 2001. 

\bibitem[GKZ94]{GKZ}{}
I. Gelfand,  M. Kapranov, A. Zelevinsky. Discriminants, resultants and multidimensional determinants. Reprint of the 1994 edition. Modern Birkh\"{a}user Classics. Birkh\"{a}user Boston, Inc., Boston, MA, 2008.

\bibitem[H-L12]{HL12}{}
D. Halpern-Leistner. {\em The derived category of a GIT quotient}. \href{http://arxiv.org/pdf/1203.0276.pdf}{arXiv:1203.0276}.

\bibitem[HHP08]{HHP}{}
M. Herbst, K. Hori, D. Page. \emph{Phases Of N=2 Theories In 1+1 Dimensions With Boundary}. \href{http://arxiv.org/pdf/0803.2045.pdf}{arXiv:0803.2045}.
%
\bibitem[HW12]{HW}
M. Herbst, J. Walcher. \emph{On the unipotence of autoequivalences of toric complete intersection Calabi-Yau categories}. Math. Ann. 353 (2012), no. 3, 783-802.
%
\bibitem[HTo07]{HTo}{}
K. Hori, D. Tong. \emph{Aspects of Non-Abelian Gauge Dynamics in Two-Dimensional N=(2,2) Theories}. J. High Energy Phys. 2007, no. 5, 079, 41 pp.

\bibitem[Hor11]{Hori}{}
K. Hori. \emph{Duality In Two-Dimensional (2,2) Supersymmetric Non-Abelian Gauge Theories}. \href{http://arxiv.org/pdf/1104.2853.pdf}{arXiv:1104.2853}.
% 
\bibitem[HTa11]{HTa1}{}
S. Hosono, H. Takagi. \emph{Mirror symmetry and projective geometry of Reye congruences I}. \href{http://arxiv.org/pdf/1101.2746.pdf}{arXiv:1101.2746}.

\bibitem[HTa13a]{HTa2}{}
S. Hosono, H. Takagi. \emph{Duality between $\text{Chow}^2$ $\P^4$ and the Double Quintic Symmetroids}. \href{http://arxiv.org/pdf/1302.5881.pdf}{arXiv:1302.5881}.

\bibitem[HTa13b]{HTa3}{}
S. Hosono, H. Takagi. \emph{Double quintic symmetroids, Reye congruences, and their derived equivalence}. \href{http://arxiv.org/pdf/1302.5883.pdf}{arXiv:1302.5883}.

% \bibitem[HK91]{HK91}
% J. Huebschmann, T. Kadei\u{s}hvili. {\em Small models for chain algebras}. Math. Z. 207 (1991), no. 2, 245-280. 

% \bibitem[Hue10]{Hue10}
% J. Huebschmann. {\em On the construction of $A_{\infty}$-structures}. Georgian Math. J. 17 (2010), no. 1, 161-202.

% \bibitem[Hes79]{Hess}{}
% W. Hesselink. \emph{Desingularizations of varieties of nullforms}. Invent. Math. 55 (1979), no. 2, 141-163.

\bibitem[IK04]{IgKi}
K. Igusa. {\em Graph cohomology and Kontsevich cycles}. Topology 43 (2004), no. 6, 1469-1510.

\bibitem[Isi12]{Isik}{}
M. U. Isik. \emph{Equivalence of the derived category of a variety with a singularity category}. Int Math Res Notices (2012), \href{http://imrn.oxfordjournals.org/content/early/2012/05/03/imrn.rns125.abstract}{doi:10.1093/imrn/rns125}

% \bibitem[Kad80]{Kad80}
% T. Kadei\u{s}vili. {\em On the theory of homology of fiber spaces}. (Russian) International Topology Conference (Moscow State Univ., Moscow, 1979). Uspekhi Mat. Nauk 35 (1980), no. 3 (213), 183-188; translation in Russian
% Math. Surveys, 35 (1980), no. 3, 231-238.

% \bibitem[Kad82]{Kad82}
% T. Kadei\u{s}vili. {\em The algebraic structure in the homology of an A($\infty$)-algebra}. (Russian) Soobshch. Akad. Nauk Gruzin. SSR 108 (1982), no. 2, 249-252. 

% \bibitem[Kad88]{Kad88}
% T Kadei\u{s}hvili. {\em The structure of the A($\infty$)-algebra, and the Hochschild and Harrison cohomologies}. (Russian) Trudy Tbiliss. Mat. Inst. Razmadze Akad. Nauk Gruzin. SSR 91 (1988), 19-27.

% \bibitem[Kaw02]{KawD-K}{}
% Y. Kawamata. {\em $D$-equivalence and $K$-equivalence}. J. Differential Geom. 61 (2002), no. 1, 147-171.

\bibitem[Kaw02]{KawFF}
Y. Kawamata. {\em Francia's flip and derived categories}. Algebraic geometry, 197-215, de Gruyter, Berlin, 2002.
%
%\bibitem[Kaw05]{Kaw05}{}
%Y. Kawamata. \emph{Log crepant birational maps and derived categories}. J. Math. Sci. Univ. Tokyo 12 (2005), no. 2, 211-231.
%
%\bibitem[Kaw06]{Kaw06}{}
%Y. Kawamata. \emph{Derived categories of toric varieties}. Michigan Math. J. 54 (2006), no. 3, 517-535.
%
%\bibitem[Kaw12]{Kaw12}{}
%Y. Kawamata. \emph{Derived categories of toric varieties, II}. \href{http://arxiv.org/pdf/1201.3460.pdf}{arXiv:1201.3460}.
% 
% \bibitem[Kel94]{Kel94}
% B. Keller. {\em Deriving DG categories}. Ann. Sci. \'Ecole Norm. Sup. (4) 27 (1994), no. 1, 63-102.

% \bibitem[Kel99]{Kel99}
% B. Keller. {\em On the cyclic homology of exact categories}. J. Pure Appl. Algebra 136 (1999), no. 1, 1-56.

%\bibitem[Kem78]{Kempf}{}
%G. Kempf. \emph{Instability in invariant theory}. Ann. of Math. (2) 108 (1978), no. 2, 299-316.

% \bibitem[Kir84]{Kir}{}
% F. Kirwan. Cohomology of quotients in symplectic and algebraic geometry. Mathematical Notes, 31. Princeton University Press, Princeton, NJ, 1984.

% \bibitem[Kon94]{Kon94} 
% M. Kontsevich. {\em Homological algebra of mirror symmetry}. Proceedings of the International Congress of Mathematicians, Vol. 1, 2 (Z\"urich, 1994), 120-139, Birkh\"auser, Basel, 1995. 

% \bibitem[Kon05]{Kon05}
% M. Kontsevich. \emph{Noncommutative motives.} Talk at the Institute for Advanced Study on the occasion of the $61^{st}$ birthday of Pierre Deligne, October 2005. Video available at \href{http://video.ias.edu/Geometry-and-Arithmetic}{video.ias.edu/Geometry-and-Arithmetic}.

% \bibitem[Kon09]{Kon09}{}
% M. Kontsevich. \emph{Notes on motives in finite characteristic}. Algebra, arithmetic, and geometry: in honor of Yu. I. Manin. Vol. II, 213-247, Progr. Math., 270, Birkh\"auser Boston, Inc., Boston, MA, 2009.

\bibitem[KS01]{KS}
M. Kontsevich, Y. Soibelman. {\em Homological mirror symmetry and torus fibrations}. Symplectic geometry and mirror symmetry (Seoul, 2000), 203-263, World Sci. Publ., River Edge, NJ, 2001.

\bibitem[Kuz05]{Kuz05}{}
A. Kuznetsov. \emph{Derived categories of quadric fibrations and intersections of quadrics}. Adv. Math. 218 (2008), no. 5, 1340-1369.

\bibitem[Kuz06]{Kuz06}{}
A. Kuznetsov. \emph{Homological projective duality for Grassmannians of lines}. \href{http://arxiv.org/abs/math/0610957}{arXiv:math/0610957}.

\bibitem[Kuz07]{KuzHPD}{}
A. Kuznetsov. \emph{Homological projective duality}. Publ. Math. Inst. Hautes \'{E}tudes Sci. No. 105 (2007), 157-220.

\bibitem[Kuz09]{Kuz09}
A. Kuznetsov. {\em Hochschild homology and semi-orthogonal decompositions}. \href{http://arxiv.org/abs/math/0904.4330}{arXiv:0904.4330}.

% \bibitem[LH03]{HL}
% K. Lef\`evre-Hasegawa. {\em Sur les A-infini cat\'egories}. \href{http://arxiv.org/abs/math/0310337}{arXiv:math/0310337}.

% \bibitem[LP11]{LP}{}
% K. Lin, D. Pomerleano. \emph{Global matrix factorizations}. \href{http://arxiv.org/abs/1101.5847}{arXiv: 1101.5847}.

\bibitem[LH03]{lefevre-hasegawa}
K.~Lef\`evre-Hasegawa.
\newblock {\em {Sur les A-infini cat\'egories}}.
\newblock PhD thesis, {Universit\'e Paris Diderot - Paris 7}, 2003.

\bibitem[Mar04]{Ma04}
M.~Markl. {\em Transferring $A\sb \infty$ (strongly homotopy associative) structures}. Rend. Circ. Mat. Palermo (2) Suppl. No. 79 (2006), 139-151.

% \bibitem[MFK94]{MFK}
% D. Mumford, J. Fogarty, F. Kirwan. Geometric invariant theory. Third edition. Ergebnisse der Mathematik und ihrer Grenzgebiete (2) [Results in Mathematics and Related Areas (2)], 34. Springer-Verlag, Berlin, 1994.

\bibitem[Nee92]{Nee}
A. Neeman. {\em The connection between the $K$-theory localization theorem of Thomason, Trobaugh and Yao and the smashing subcategories of Bousfield and Ravenel}. Ann. Sci. \'Ecole Norm. Sup. (4) 25 (1992), no. 5, 547-566.

% \bibitem[Nes84]{Ness}{}
% L. Ness. \emph{A stratification of the null cone via the moment map}. With an appendix by D. Mumford. Amer. J. Math. 106 (1984), no. 6, 1281-1329.

% \bibitem[Orl92]{Orl92}{}
% D. Orlov. \emph{Projective bundles, monoidal transformations, and derived categories of coherent sheaves}. Izv. Ross. Akad. Nauk Ser. Mat. 56 (1992), no.4, 852-862.

\bibitem[Orl04]{Orl04}{}
D. Orlov. \emph{Triangulated categories of singularities and {D}-branes in {L}andau-{G}inzburg models}. Tr. Mat. Inst. Steklova 246 (2004), 240-262.
% 
% \bibitem[Orl05]{Orl05}{}
% D. Orlov. \emph{Derived categories of coherent sheaves and motives}. Russian Math. Surveys 60 (2005), 1242-1244.
% 
% \bibitem[Orl06]{Orl06}{}
% D. Orlov, \emph{Triangulated categories of singularities and equivalences between {L}andau-{G}inzburg models}. (Russian) Mat. Sb. 197 (2006), no. 12, 117-132; translation in Sb. Math. 197 (2006), no. 11-12, 1827-1840.
% 
\bibitem[Orl09]{Orl09}{}
D. Orlov. \emph{Derived categories of coherent sheaves and triangulated categories of singularities}. Algebra, arithmetic, and geometry: in honor of Yu. I. Manin. Vol. II, 503-531, Progr. Math., 270, Birkh\"{a}user Boston, Inc., Boston, MA, 2009.

\bibitem[Orl11]{OrlMF}{}
D. Orlov. \emph{Matrix factorizations for nonaffine LG models}. Math. Ann. 353 (2012), no. 1, 95-108.

% \bibitem[P05]{Pol} 
% A. Polishchuk, {\em Noncommutative proj and coherent algebras}. Math. Res. Lett. 12 (2005), no. 1, 63-74. 

\bibitem[Pos09]{Pos1}{}
L. Positselski. \emph{Two kinds of derived categories, Koszul duality, and comodule-contramodule correspondence}. \href{http://arxiv.org/abs/0905.2621}{arXiv:0905.2621}.

\bibitem[Pos11]{Pos2}{}
L. Positselski. \emph{Coherent analogues of matrix factorizations and relative singularity categories}. \href{http://arxiv.org/abs/1102.0261}{arXiv:1102.0261}.

% \bibitem[PV10]{PV2}{}
% A. Polishchuk, A. Vaintrob. \emph{Matrix factorizations and singularity categories for stacks}. \href{http://arxiv.org/pdf/1011.4544}{arXiv:1011.4544}.

% \bibitem[Res00]{Res}{}
% N. Ressayre. \emph{The GIT-equivalence for $G$-line bundles}. Geom. Dedicata 81 (2000), no. 1-3, 295-324.

\bibitem[Ric10]{Riche}
S. Riche. {\em Koszul duality and modular representations of semisimple Lie algebras}. Duke Math. J. 154 (2010), no. 1, 31-134.

\bibitem[R\o{}d00]{Rod}{}
E. A. R\o{}dland. \emph{The Pfaffian Calabi-Yau, its mirror, and their link to the Grassmannian $G(2,7)$}. Compositio Math. 122 (2000), no. 2, 135-149. 

%\bibitem[Sch06]{Schlich}
%M. Schlichting. {\em Negative $K$-theory of derived categories}. Math. Z. 253 (2006), no. 1, 97-134.

% \bibitem[Sei08]{SeiBook}
% P. Seidel. Fukaya categories and Picard-Lefschetz theory. Zurich Lectures in Advanced Mathematics. European Mathematical Society (EMS), Z\"urich, 2008. 

\bibitem[Sei11]{Se} 
P. Seidel. \emph{Homological mirror symmetry for the genus two curve}. J. Algebraic Geom. 20 (2011), no. 4, 727-769.

\bibitem[Seg11]{Seg2}{}
E. Segal. \emph{Equivalences between GIT quotients of Landau-Ginzburg B-models}. Comm. Math. Phys. 304 (2011), no. 2, 411-432.

% \bibitem[Ser55]{SerreFAC}
% J.-P. Serre. {\em Faisceaux alg\'ebriques coh\'erents}. (French) Ann. of Math. (2) 61 (1955), 197-278.

\bibitem[Sha10]{Sha}
E. Sharpe. {\em Landau-Ginzburg models, gerbes, and Kuznetsov's homological projective duality}. Superstrings, geometry, topology, and $C^*$-algebras, 237-249, Proc. Sympos. Pure Math., 81, Amer. Math. Soc., Providence, RI, 2010.

% \bibitem[Sha12]{Sha}{}
% E. Sharpe, \emph{GLSMs, gerbes, and Kuznetsov's homological projective duality,} Proceedings of Quantum theory and symmetries 6 (2012).

\bibitem[Shi12]{Shipman}{}
I. Shipman. {\em A geometric approach to Orlov's theorem}. Compositio Math. 148 (2012), no. 5, 1365-1389.

% \bibitem[Tab12]{Tab}{}
% G. Tabauda. \emph{Chow motives versus noncommutative motives}. To appear in the Journal of Noncommutative Geometry.

% \bibitem[Tev05]{Tev05}
% E.~A. Tevelev. Projective duality and homogeneous spaces. Encyclopaedia of Mathematical Sciences, 133. Invariant Theory and Algebraic Transformation Groups, IV. Springer-Verlag, Berlin, 2005. 

\bibitem[Tha96]{Tha96}{}
M. Thaddeus. \emph{Geometric invariant theory and flips}. J. Amer. Math. Soc. 9 (1996), no. 3, 691-723. 

\bibitem[Tho97]{Tho2}{}
R.W. Thomason. \emph{Equivariant resolution, linearization, and Hilbert's fourteenth problem over arbitrary base schemes}. Adv. Math. 65 (1987), 16-34.

% \bibitem[To\"e07]{Toe07}{}
% B. To\"en. \emph{The homotopy theory of dg-categories and derived Morita theory}. Invent. Math. 167 (2007), no. 3, 615-667. 

\bibitem[VdB04]{VdB}{}
M. Van den Bergh. {\em Non-commutative crepant resolutions}. The legacy of Niels Henrik Abel, 749-770, Springer, Berlin, 2004.

% \bibitem[Ver77]{Verdier}
% J.-L. Verdier. {\em Cat\'egories d\'eriv\'ees, \'etat $0$}. in: SGA $4\frac{1}{2}$, 262-312, Lecture Notes in Math., vol. 569. Springer, Berlin, 1977.

\end{thebibliography}
\end{document}